\documentclass[a4paper, reqno, 11pt]{amsart}
\makeatletter

\@addtoreset{equation}{section}
\makeatother
\usepackage{amsmath}
\usepackage{amssymb}
\usepackage{amsthm}
\usepackage[dvipdfmx]{graphicx}
\usepackage[dvipdfm]{hyperref}
\usepackage{tikz}
\newtheorem{Thm}{Theorem}[section]

\newtheorem{Lem}[Thm]{Lemma}
\newtheorem{Prop}[Thm]{Proposition}

\newtheorem{Cor}[Thm]{Corollary}

\theoremstyle{definition}

\newtheorem{Def}[Thm]{Definition}
\newtheorem{Rem}[Thm]{Remark}

\newcommand{\Z}{\mathbb{Z}} 
\newcommand{\Pp}{\mathbb{P}_p}
\newcommand{\R}{\mathbb{R}}
\newcommand{\ca}{\textup{Cap}} 
\newcommand{\va}{\textup{Var}}

\begin{document}

\title[Union of random walk and percolation]{Unions of random walk and percolation on infinite graphs}
\author{Kazuki Okamura}
\address{School of General Education, Shinshu University, 3-1-1, Asahi, Matsumoto, Nagano, JAPAN} 
\email{kazukio@shinshu-u.ac.jp} 

\begin{abstract}
We consider a random object that is associated with both random walks and random media, specifically, 
the superposition of a configuration of subcritical Bernoulli percolation on an infinite connected graph and the trace of the simple random walk on the same graph.
We investigate asymptotics for the number of vertices of the enlargement of the trace of the walk until a fixed time, when the time  tends to infinity. 
This process is more highly self-interacting than the range of random walk, which yields difficulties.    
We show a law of large numbers on vertex-transitive transient graphs. 
We compare the process on a vertex-transitive graph with the process on a finitely modified graph of the original vertex-transitive graph and show their behaviors are similar.     
We show that the process fluctuates almost surely on a certain non-vertex-transitive graph. 
On the two-dimensional integer lattice, 
by investigating the size of the boundary of the trace,  
we give an estimate for variances of the process implying a law of large numbers.   
We give an example of a graph with unbounded degrees on which the process behaves in a singular manner.  
As by-products, some results for the range and the boundary, which will be of independent interest, are obtained.   
\end{abstract}

\maketitle

\setcounter{tocdepth}{1}
\tableofcontents

\section{Introduction and Main Results}
  
Consider Bernoulli bond percolation on an infinite connected graph $G$. 
Assume that each edge of $G$ is open with probability $p \in [0,1]$ and closed with probability $1-p$. 
It seems natural to consider the following informal question: if we add Bernoulli percolation on $G$ to a subgraph $H$ of $G$, 
then how much does $H$ change? 
By this motivation, 
the author \cite{O} proposed a model in which a configuration of Bernoulli percolation on an infinite connected graph $G$ is added to a (deterministic or random) subgraph $H$ independently, 
and then, asked whether the probability that a property $\mathcal{P}$ of $H$ remains to be satisfied for the enlargement of $H$ is less than $1$, as $p$ increases.   
If $H$ is a single vertex of $G$ and $\mathcal{P}$ is the property that the graph has an infinite number of vertices, 
we obtain the definition of Hammersley's critical probability. 
In \cite{O}, an important example of such a subgraph $H$ and a property $\mathcal{P}$ is the case that $H$ is the trace of the simple random walk $\{S_n\}_n$ on $G$ and $\mathcal{P}$ is that the (enlarged) graph is recurrent, that is, the simple random walk on the graph is recurrent. 

In this paper, 
we terminate the simple random walk on an infinite connected simple graph at a time $n$, and 
consider asymptotics for the number of vertices of the enlargement of the trace {\it until the time $n$} by subcritical Bernoulli percolation on the same graph.    
The main focus of \cite{O} is the case that $H$ is infinite.  
On the other hand, we focus on the case that $H = H_n$ is {\it finite}, and Furthermore depends on the time $n$.  
We denote by $U_n$ the number of the vertices of the union of the simple random walk until time $n$ and Bernoulli percolation (see Definition \ref{union}). 
Our purpose is to investigate asymptotics for $\{U_n\}_n$.  
We compare $\{U_n\}_n$ with $\{R_n\}_n$, which is the number of the vertices that are visited by the simple random walk until time $n$, and how much the behaviors of $\{U_n\}_n$ depend on the parameter $p$ of Bernoulli percolation. 
If we add no percolation clusters (i.e. $p = 0$),  
then $U_n$ is identical with $R_n$.  
If we add non-trivial percolation clusters  (i.e. $p > 0$), then $\{U_n\}_n$ becomes more complicated. 
$\{U_n\}_n$ is more highly self-interacting than $\{R_n\}_n$, which is already self-interacting.  
Such self-interacting nature yields difficulties.   

This process is associated with both random walks and random media.  
Of such objects, random walk in random environment (RWRE), including random walk on percolation cluster, is well-known and has been intensively studied. 
Random walk in random scenery (RWRS) is another known model. 
It is a random process such that at each time, both the step taken by the walk and the scenery value at the site that is visited are registered. 
To the best of our knowledge, our framework is different from any known studies on RWRE and on RWRS. 

We first show a strong law of large numbers (LLN) for $\{U_n\}_n$ on vertex-transitive graphs.  
The almost sure limit of $U_n/n$ depends on the parameter of the percolation $p$ and is denoted by $c_p$. 
We consider properties of $c_p$ as a function of $p$. 
Specifically, we show that $c_p$ is analytic and has at least linear growth. 
We consider how much the behavior of $\{U_n\}_n$ changes if we modify $G$.  
Long-time behaviors of $\{U_n / n\}_n$ and $\{E[U_n] -c_p n\}_n$ are stable with respect to finite modifications of vertex-transitive transient graphs. 
(See Definition \ref{def-fm} for the definition of finite modification.)      
It is shown that $\{U_n / n\}_n$ fluctuates almost surely on a certain non-vertex-transitive transient graph. 
On the two-dimensional integer lattice, 
by investigating the size of the boundary of the trace of random walk,  
we give upper bounds for variances and LLN for $\{U_n\}_n$.  
We give an example of a recurrent graph with {\it unbounded} degrees on which $\{U_n\}_n$ behaves in a ``singular" manner.    
We consider positive and negative exponentials of $\{U_n\}_n$, and obtain results similar to Hamana \cite{H} and Donsker-Varadhan \cite{DV79} which  studied Laplace transform of the range of simple random walk. 
In the Appendix, behaviors of $\{E^{\Pp}[U_n]\}_n$ are considered. 
Here $E^{\Pp}$ is the expectation with respect to percolation. 
Roughly speaking, 
$E^{\Pp}[U_n]$ is the process obtained by eliminating the randomness of percolation.  
It is somewhat easier to analyze than $\{U_n\}_n$.    

Many results are similar to those for the range $\{R_n\}_n$. 
However, our proofs are different from those for the corresponding results, because  we need to deal with highly self-interacting nature of $\{U_n\}_n$. 
In Corollary \ref{disconti} below, we state that a discontinuity occurs in the ratio of the logarithm of the Laplace transforms of $U_n$ and $R_n$ for $\Z^3$.
As by-products of this research, 
we obtain some results for the range and the size of the boundary of the trace, which will be of independent interest.   

\subsection{Framework and Notation}  

We introduce some notation.  
The expectation with respect to a probability measure $\mu$ is denoted by $E^{\mu}$.
For two probability measures $\mu$ and $\nu$, $\mu \otimes \nu$ denotes the product measure of $\mu$ and $\nu$.   
We let the infimum of an emptyset $\inf \emptyset$ be $+\infty$. 
For two functions $f, g$ on the integers, 
$f(n) \simeq g(n)$ means that there are some constants $0 < c < C < +\infty$ such that $c|g(n)| \le |f(n)| \le C|g(n)|$ for any large $n$.  
$f(t) \sim g(t)$ means that $f(t)/g(t) \to 1$ as either $t \to 0$ or $t \to \infty$.   
Readers will immediately see which of the two cases of the limit are considered, case by case.

Let $G = (V(G), E(G))$ be an infinite connected simple graph. 
For ease of notation, we often denote $V(G)$ by $G$.  
{\it We assume that it has bounded degrees, unless stated otherwise.}
Let $\Delta_G$ be the maximal degree of $G$.          
{\it If we give a subset $H$ of $V(G)$ and do not refer the set of edges, then, the graph considered is the induced subgraph of $H$. }
Let $d$ be the graph distance of $G$ and $B(x, n) := \{y \in V(G) : d(x,y) \le n\}$. 
For $A \subset G$, denote the cardinality of $A$ by $|A|$, and the complement of $A$ by $G \setminus A$ or $A^c$. 
Let $\textup{diam}(A)$ be the diameter of $A$, that is, the supremum of the distance between two points in $A$. 
In this paper, if $G = \Z^d$ and the set of edges are not referred, then it is the nearest-neighbor model, that is, the set of edges is the collection of two adjacent vertices of $G$. 
Let $| \cdot |_{\infty}$ be the infinity norm of $\Z^d$ and $B_{\infty}(x,r)$ be the open ball having center $x$ and radius $r$ with respect to $| \cdot |_{\infty}$.  
For disjoint subsets $A$ and $B$ of $V(G)$, 
the effective resistance $R_{\textup{eff}}(A, B)$ between $A$ and $B$ is defined by 
\[ R_{\textup{eff}}(A, B)^{-1} := \inf\left\{ \sum_{\{x,y\} \in E(G)} \left(f(x) - f(y) \right)^2 : f = 1 \textup{ on } A, \ f = 0 \textup{ on } B \right\}.  \]

Let $(S_{n})_{n \ge 0}$ be the simple random walk  on $G$. 
Let $P^{x}$ be the law of $(S_{n})_{n}$ starting at $x \in G$. 
We say that $G$ is {\it transient} or {\it recurrent} if the simple random walk on $G$ is transient or recurrent, respectively.  
Denote $\{S_m, \dots, S_n\}$ by $S[m, n]$. 
We call it the trace of random walk. 
Define the random walk range up to time $n$ by 
\[ R_{n} := \left| S[0,n] \right| = \left| \left\{S_0, \dots, S_n \right\}\right|.\]  
We define 
\[ T_{A} := \inf\left\{n \ge 1 : S_{n} \in A \right\} \textup{ and }  H_A := \inf\left\{n \ge 0 : S_n \in A \right\}, \ \ A \subset G. \]   

Consider Bernoulli {\it bond} percolation on $G$.
Let $\Pp$ be the Bernoulli measure with parameter $p$. 
Let $C_{x}$ be the open cluster containing a vertex $x$ of $G$.  
Let $p_{T}(G)$ be Temperley's critical probability, 
that is, 
\[ p_{T}(G) := \inf\left\{ p \in [0,1] : E^{\Pp}[|C_{o}|] = +\infty \right\}. \]
This value does not depend on the choice of $o$. 
We have that $p_{T}(G) \ge 1/(\Delta_G - 1)$. 
(See Bollob\'as and Riordan \cite[Chapter 1]{BR}, for example.)   
{\it  In all assertions in this paper, we assume that $p < p_{T}(G)$. } 

Let $o$ be a vertex of $G$. 
Let $\widetilde P^{o,p}$ be the product measure $P^o \otimes \Pp$ of $P^o$  and $\Pp$. 
Precisely, 
$P^x$ is a probability measure on the path space of random walks on $G$, and $\widetilde P^{o,p}$ is a probability measure on the product space of the path space of the random walks on $G$ 
and $\{0,1\}^{E(G)}$. 
Let $\widetilde E^{o,p}$ be $E^{P^o \otimes \Pp}$, and let $\va_{o, p}$ be the variance with respect to $P^o \otimes \Pp$.

We say that $G$ is  {\it vertex-transitive} if the number of the equivalent classes of $G$ is exactly one, 
where $x$ and $y$ are equivalent if there is a graph automorphism $\gamma$ of $G$ such that $\gamma(x) = y$.  
If $G$ is vertex-transitive, then, the law of $\widetilde P^{o,p}$ do not depend on $o$, and hence 
we drop the $o$ and write $\widetilde P^{p} = \widetilde P^{o,p}$, $\widetilde E^{p} = \widetilde E^{o,p}$, $E = E^{P^o}$ and $\va_{p} = \va_{o, p}$.  
We clarify the dependence of these on $G$ if needed.  

\begin{Def}[Volumes of unions of random walk  and percolation]\label{union}
\[ U_{n} := \left|  \bigcup_{x \in \{S_0, \dots, S_n\}} C_x \right|, \ \ n \ge 0. \]   
\end{Def}

Here, $U_n$ is increasing  with respect to $n$. 
Since $x \in C_x$,  
it holds that $U_{n} \ge R_{n}$. 
If $p=0$, then $U_{n} = R_{n}$, $\widetilde P^{o,p}$-a.s.  

\begin{Def}\label{def-fm} 
We say that $G^{\prime}$ is a {\it finite modification} of $G$ 
if there exist two finite subsets $D$ on $G$ and $D^{\prime}$ on $G^{\prime}$ 
such that there is an isomorphism (see Diestel  \cite[Section 1.1]{D} for the definition of this terminology) $\phi : G \setminus D  \to G^{\prime} \setminus D^{\prime}$.  
We will see that $G^{\prime}$ is roughly isometric to $G$. (See Woess \cite[Definition 3.7]{W} for the definition of being roughly isometric.)  
\end{Def}  

\subsection{Results for transient graphs} 

\begin{Thm}[Law of large numbers]\label{DisLLN}
Assume that $p < p_T (G)$. 
Let $G$ be a vertex-transitive graph, and $o$ be a vertex of $G$. 
Let 
\begin{equation}\label{cp}
c_p = c_{G, p} := E^{\Pp}\left[|C_{o}| P^{o}\left(T_{C_{o}} = +\infty\right)\right].   
\end{equation} 
Then, for any $1 \le q < +\infty$, 
\begin{equation}\label{LLN} 
\lim_{n \to \infty} \frac{U_{n}}{n} = c_{p}, \ \textup{ $\widetilde P^{o, p}$-a.s. and in $L^{q}(\widetilde P^{o, p})$}.
\end{equation}  
\end{Thm}

As we see in Theorem \ref{fluc}, 
we cannot define $c_p$ for a certain non-vertex-transitive graph.  
If $G$ is transient, then,  by \eqref{LLN}, 
\[ c_p \ge c_{0} = P^{o}\left(T_{o} = +\infty\right) > 0. \]  
On the other hand, if $G$ is recurrent, then $c_p = 0$ for any $p$. 

We will show Theorem \ref{DisLLN} 
by applying Liggett's subadditive ergodic theorem to 
\begin{equation}\label{Umn-Def} 
U_{m,n} := \left|\bigcup_{x \in \{S_m, \dots, S_n\}} C_{x}\right|. 
\end{equation} 
Informally speaking, we will show that for each $l$, 
$U_{0, l}$ and $U_{kl, (k+1)l}$ are asymptotically independent as $k$ tends to $\infty$. 

If $G = \Z^d, d \ge 3$, and $p = 0$, 
then, this assertion was shown by Dvoretzky-Erd\"os \cite[Theorem 4]{DE}. 
Spitzer \cite[Section 4]{S64} stated that the strong law of the volume of a discrete analog for the Wiener sausage can be shown, in the same manner as in the continuous case. 
We will deal with this process in the Appendix. 
However, more delicate arguments would be required for $\{U_n\}_n$, because it is highly self-interactive. 
Because of the high self-intersecting nature, it is  interesting to establish a central limit theorem for $\{U_n\}_n$ on $\Z^d, d \ge 3$. 

We now consider properties for $(c_p)_{p  \in [0, p_{T}(G))}$, 
which is somewhat similar to those for $E^{\Pp}\left[|C_o|\right]$ as function of $p$. 

\begin{Thm}[Properties of $c_p$]\label{cp-property}
Assume that $p < p_T (G)$. 
(i) If $G$ is vertex-transitive and transient, then, $c_p$ is analytic on $p \in [0, p_{T}(G))$.\\
(ii) Let $G$ be a Cayley graph of a finitely generated infinite group and assume it is transient.   
Then, 
\begin{equation}\label{cp-modc}
\frac{d}{dp} c_p > 0, \ \ 0 < p < p_T(G).    
\end{equation} 
(iii) If $G = \Z^d, d \ge 11$, then 
\[  \lim_{p \to p_{T}(\Z^d)} c_p = +\infty. \]
\end{Thm} 

We say that a volume growth condition $V(d)$ holds 
if there is a positive constant $C$ such that 
\[ |B(x,n)| \ge C n^{d}, \ x \in G, \ n \ge 1. \] 

\begin{Thm}[Finite modification]\label{fm} 
Assume that $p < p_T (G)$. 
Assume that $G$ is a vertex-transitive graph and $G^{\prime}$ is a finite modification of $G$.  
Then, \\
(i) For any vertex $o$ of $G^{\prime}$, 
\begin{equation}\label{fm-as} 
\lim_{n \to \infty} \frac{U_n}{n} = c_{G,p}, \textup{ $\widetilde P_{G^{\prime}}^{o,p}$-a.s.}  
\end{equation} 
If $G$ satisfies  $V(d)$ for some $d > 2$, 
\begin{equation}\label{Lv1-conv} 
\lim_{n \to \infty} \frac{\widetilde E_{G^{\prime}}^{o, p}[U_n]}{n} = c_{G, p},   
\end{equation} 
and,  
\begin{equation}\label{Lv1-conv-L1} 
\lim_{n \to \infty} \frac{U_n}{n} = c_{G, p}, \textup{ in  $L^1 (\widetilde P_{G^{\prime}}^{o,p})$. } 
\end{equation} 
(ii) If $G$ satisfies  $V(d)$ for some $d > 4$, 
then, for any vertex $o$ of $G^{\prime}$, 
\[ \lim_{n \to \infty} \widetilde E_{G^{\prime}}^{o,p}[U_n] - c_{G,p}n \] 
exists. 
The limit does not take $\pm \infty$.      
\end{Thm}

By \cite[Lemma 3.12]{W}, 
any vertex-transitive transient graph $G$ satisfies $V(2)$, 
so, we believe that the assumption of assertion (i) is not a large restriction. 
Assertions (i) and (ii) are applicable to Cayley graphs of finitely generated group having polynomial volume growth with degree $d \ge 3$, and $d \ge 5$, respectively. 

We now leave the case that $G$ is vertex-transitive.  
If a transient graph $G$ is {\it not} vertex-transitive, then \eqref{LLN} can {\it fail} in the following sense.  

\begin{Thm}[{$E[U_n]/n$} can fluctuate on a transient graph which is not vertex-transitive]\label{fluc} 
Assume that $p < p_T (G)$. 
There is a graph $G$ and a vertex $o$ of $G$ such that for any $p \in [0, p_T (G))$, 
the following holds $\widetilde P^{o,p}$-a.s.  
\begin{equation}\label{fluc-fluc}
\liminf_{n \to \infty} \frac{\widetilde E^{o,p} [U_{n}]}{n} = \liminf_{n \to \infty} \frac{U_{n}}{n} < \limsup_{n \to \infty} \frac{U_{n}}{n} = \limsup_{n \to \infty} \frac{\widetilde E^{o,p} [U_{n}]}{n}.  
\end{equation}  
\end{Thm}

In this case, we {\it cannot} define the value $c_p$ in \eqref{cp}.      
If $p = 0$, then, this assertion extends the author's paper \cite[Theorem 1.3]{O14}. 
We use the convergence result of Theorem \ref{fm} in order to show this fluctuation result. 
It is more interesting to find a necessary and sufficient condition for \eqref{fluc-fluc}. 


\subsection{Results for recurrent graphs}

\begin{Thm}\label{lim-rec} 
Assume that $p < p_T (G)$. 
If $G$ is recurrent and vertex-transitive, then,  
\begin{equation}\label{lim-rec-0} 
\lim_{n \to \infty} \frac{\widetilde E^{p}[U_{n}]}{E[R_n]} = 1. 
\end{equation}  
\end{Thm} 

A certain {\it homogeneity} assumption of $G$ would be crucial for \eqref{lim-rec-0},  
because we can give an example of an inhomogeneous graph on which \eqref{lim-rec-0} fails. 
See Remark \ref{Rem-unbdd} (ii). 

\begin{Thm}\label{lowdim-dis} 
If $G = \Z^2$, then, \\
(a) 
\begin{equation}\label{asy-2d}
\left|\widetilde E^p [U_n] - \frac{n}{\log n} \pi \right| = O\left(\frac{n}{(\log n)^2}\right). 
\end{equation}
(b) 
\begin{equation}\label{Var-2d}
\va_p (U_n) = O\left(\frac{n^2}{(\log n)^4}\right).  
\end{equation}
(c) 
For any $1 \le q < +\infty$, 
\begin{equation}\label{lim-2-dis-as} 
\lim_{n \to \infty} \frac{\log n }{n} U_{n} =  \pi,  \ \textup{ $\widetilde P^{p}$-a.s. and in $L^q (\widetilde P^{p})$.}  
\end{equation}  
\end{Thm}

If $p = 0$, then, \eqref{asy-2d}, \eqref{Var-2d}, and \eqref{lim-2-dis-as} were obtained by Jain-Pruitt \cite[Lemma 3.1]{JP70}, \cite[Theorem 4.2]{JP72}\footnote{They showed $(\log n)^4 \va(R_n) / n^2$ converges to a positive constant as $n \to \infty$ for a general class of random walk.}, and \cite[Theorem 4]{DE}, respectively. 
Le Gall \cite[Lemme 6.2]{Le86CMP} also shows \eqref{Var-2d} for the case of $p = 0$, by using the estimate for intersections of two independent random walks.   
For $n_1 < n_2 \le n_3 < n_4$, $U_{n_1, n_2}$ and $U_{n_3, n_4}$ are not independent, which will be an obstacle to applying the method of the proof of \cite[Lemme 6.2]{Le86CMP}.  
We show (a), (b) and (c) by considering the boundary of the trace of the simple random walk, 
specifically, using the phenomenon that on a recurrent graph, the boundary of the trace is ``sufficiently" smaller than the trace.  
It is also interesting to establish a central limit theorem for $\{U_n\}_n$ on $\Z^2$. 

As the following shows, 
the behaviour of $\{U_n\}_n$ on a certain locally-finite recurrent graph with {\it unbounded} degrees is significantly different from that on vertex transitive graphs with bounded degrees. 

\begin{Thm}[Recurrent graph with unbounded degrees]\label{unbdd} 
There is a graph $G$  with unbounded degrees  and a vertex $o$ of $G$ such that\\
(a) $p_T (G) = 1$, \\
and Furthermore the following hold for any $p \in (0,1)$:\\
(b)
\begin{equation}\label{loglogUn} 
\limsup_{n \to \infty} \frac{\log\log U_{n}}{\log n} \ge \frac{1}{2},  \  \textup{ $\widetilde P^{o, p}$-a.s.}     
\end{equation}  
(c) 
\begin{equation}\label{VarUn} 
\limsup_{n \to \infty} \frac{\log \va_{o,p} (U_n)}{n} \ge \log \frac{1}{2p^2}.     
\end{equation}
\end{Thm}

In (b) and (c) above, the case that $p = 0$ are excluded.   
It is also interesting to find examples of (non vertex transitive) graphs with {\it bounded} degrees such that $\{U_n\}_n$ behaves in a singular manner.  

\subsection{Positive and negative exponentials}

We now consider positive exponentials of $\{U_n\}$ in the case that $G$ is vertex-transitive.
For $\theta > 0$ and $p \in [0, p_T (G))$, 
let \[ \Lambda_p (\theta) = \inf_{n \ge 1} \frac{\log E[\exp(\theta U_{n-1})]}{n}. \]
We allow to take $+\infty$ as the limit. 
We have that $\Lambda_p$ is upper semicontinuous. 

Now similarly to \cite{H},
we consider the behaviour of $\Lambda_p (\theta)$ as $\theta \to 0$. 
For each $p \in [0, p_T (G))$ and  a vertex $o$ of $G$, let 
\[ \theta_c (p) := \inf\left\{\theta > 0 : E^{\Pp}\left[\exp(\theta |C_o|)\right] = +\infty \right\}.  \]  

If $p > 0$, then,  for each $n$, 
\[ \Pp(|C_o| > n) \ge \Pp(\textup{there exists an open self-avoiding path of length $n$ starting at $o$})\]
\[ \ge p^n. \]
Hence, 
$E^{\Pp}\left[\exp(\theta |C_o|)\right] = +\infty$ for sufficiently large $\theta$. 
By \eqref{AV3}, 
it holds that 
\[ \theta_c (p) = +\infty, \ p = 0. \] 
\[ 0 < \theta_c (p) < +\infty,  \ 0 < p < p_T(G). \]
\[ \theta_c (p) = 0,   \ \ p \ge p_T(G).\] 
It holds that $\Lambda_p (\theta) < +\infty$ for $\theta \in [0, \theta_c(p))$.  

If $p = 0$, then, by the Markov property, 
\[ E^{P^0}[\exp(\theta R_{m+n})] \le E^{P^0}[\exp(\theta R_m)]E^{P^0}[\exp(\theta R_n)], \ n,m \ge 1. \] 
(see \cite{H}) and we have that 
\[ \Lambda_0 (\theta) = \lim_{n \to \infty} \frac{\log E^{P^0}[\exp(\theta R_n)]}{n}. \] 
By this and the H\"older inequality, 
$\Lambda_0 (\theta)$ is convex with respect to $\theta$, and hence, continuous on $[0, +\infty)$. 

\begin{Thm}[Positive exponentials]\label{posi} 
Assume that $p < p_T (G)$.  
(i) If $G$ is vertex-transitive and transient, 
then, \[ \lim_{\theta \to 0} \frac{\Lambda_p (\theta)}{\Lambda_0(\theta)} = \frac{c_p}{c_0},  \]
where $c_p$ is the constant in \eqref{cp}. \\
(ii) If $G = \Z^d, d = 1,2$, then,   
\[ \lim_{\theta \to 0} \frac{\Lambda_p (\theta)}{\Lambda_0(\theta)} = 1.  \] 
\end{Thm}

For $\theta \in (-\infty, \theta_c (p)) \setminus \{0\}$ and $x \in \Z^d$, 
let 
\[ \widetilde \Lambda_p (\theta) := \inf_{n \ge 1} \frac{\log \widetilde E^{p}[\exp(\theta U_{n-1})]}{\log E^{P^x}[\exp(\theta R_{n-1})]}, \ 0 < \theta  < \theta_c (p), \] 
and, 
\[ \widetilde \Lambda_p (\theta) := \sup_{n \ge 1} \frac{\log \widetilde E^{p}[\exp(\theta U_{n-1})]}{\log E^{P^x}[\exp(\theta R_{n-1})]}, \ \theta < 0. \] 
Then, 
\[ \widetilde \Lambda_p (\theta) \ge 1, \ \ 0 < \theta < \theta_c (p), \] 
and, 
\[ \widetilde \Lambda_p (\theta) \le 1, \ \theta < 0.\]
If $p = 0$, then $\widetilde \Lambda_p (\theta) = 1$ for any $\theta$.   
Since $\Lambda_p$ is upper-semicontinuous on $[0, \theta_c(p))$, $\widetilde \Lambda_p$ is so. 

By \eqref{cp-modc} and Theorem \ref{posi}, 
\begin{Cor}\label{disconti}
Let $G = \Z^d, d \ge 3$, and $0 < p < p_{T}(G)$. 
Then, 
 \[ \lim_{\theta \to 0, \theta > 0} \widetilde \Lambda_p (\theta) = \frac{c_p}{c_0}  > 1 \ge  \sup_{\theta < 0} \widetilde \Lambda_p (\theta).  \] 
Hence, there exists a discontinuity of $\widetilde \Lambda_p (\theta)$ at $\theta = 0$.  
\end{Cor}

\begin{Thm}[Negative exponentials]\label{Dis-DV} 
Assume that $p < p_T (G)$. 
If $G = \Z^d$, then, 
\begin{equation}\label{Dis-DV-int} 
\lim_{n \to \infty} \frac{-\log \widetilde E^{p} [\exp(-\theta U_n)]}{n^{d/(d+2)}} = \lim_{n \to \infty} \frac{-\log E^{P^0} [\exp(-\theta R_n)]}{n^{d/(d+2)}}. 
\end{equation} 
We have that 
\[ \widetilde \Lambda_p (\theta) = 1, \ \theta < 0.\]
\end{Thm}

It is shown in \cite{DV79} that the limit in the right hand side of \eqref{Dis-DV-int} exists and depends only on $(d, \theta)$. 
We will show this by using exponential decay of sizes of clusters in subcritical phases.  

A similar result also holds for graphs other than $\Z^d$. 
By Gibson \cite{Gi},  it is easy to see that 
if $G$ is a Cayley graph of a finitely generated group with polynomial volume growth of degree $d \ge 2$, then, 
\[ -\log \widetilde E^{p} [\exp(-\theta U_{n})] \simeq n^{d/(d+2)}. \]

\subsection{Organization of paper}

Section 2 is devoted to law of large numbers and Theorem \ref{DisLLN} is shown. 
In Section 3, we state some auxiliary results for boundary of the trace of random walk, which are used in the following sections, and then, by using them, Theorem \ref{lim-rec} is shown. 
In Section 4, properties of $c_p$ are considered, and Theorem \ref{cp-property} is shown. 
In Section 5, we deal with finite modification and fluctuations for $\{U_n\}_n$ and Theorems \ref{fm} and \ref{fluc} are shown. 
Section 6 is devoted to the proof of Theorem \ref{lowdim-dis}.   
Section 7 is devoted to the proof of Theorem \ref{unbdd} and remarks for ``one-dimensional" graphs. 
Section 8 is devoted to the proofs of Theorems \ref{posi} and \ref{Dis-DV}.   
In the Appendix, 
we consider $\{E^{\Pp}[U_n]\}_n$, which is a deterministic version of $\{U_n\}_n$.


\section{Law of large numbers} 

In this section, we show Theorem \ref{DisLLN}.   
We first consider the growth of the mean $\widetilde E^{x,p} [U_{n}]$ as $n$ tends to infinity. 

\begin{Prop}[Growth of mean]\label{growth-mean}
Assume that $p < p_T (G)$. 
For any $y \in G$ and $n \ge 1$,  
\[ (1-p)^{\Delta_G}  \sum_{i=0}^{n} \inf_{x \in G} P^x(T_x > i) \le \widetilde E^{y,p} [U_{n}] \le \sup_{x \in G} E^{\Pp}[|C_x|] \sum_{i=0}^{n} \sup_{x \in G} P^x (T_x > i). \] 
\end{Prop}

\begin{proof}
It holds that 
\begin{equation}\label{last-exit}
U_n = \sum_{0 \le i \le n} \left| C_{S_i}\right| 1_{\{C_{S_{i}} \ne C_{S_{j}}, i < \forall j \le n\}}. 
\end{equation}
We remark that $C_{S_{i}} \ne C_{S_{j}}$ is equivalent to $S_j \notin C_{S_{i}}$. 
By this, \eqref{last-exit} and the Markov property, it holds that for any $x \in G$, 
\begin{equation}\label{un-mean} 
\sum_{k = 0}^{n} \inf_{y \in G} E^{\Pp}\left[|C_{y}| P^{y}\left(T_{C_{y}} > k\right)\right] \le \widetilde E^{x, p}[U_{n}] \le \sum_{k = 0}^{n}  \sup_{y \in G} E^{\Pp}\left[|C_{y}| P^{y}\left(T_{C_{y}} > k\right)\right]. 
\end{equation} 
It is easy to see that for any $y$ and $n$, 
\[ (1-p)^{\Delta_G} \le \frac{E^{\Pp}\left[|C_{y}| P^{y}\left(T_{C_{y}} > n\right)\right]}{P^y (T_y > n)} \le \sup_{z \in G} E^{\Pp}[|C_z|]. \] 
Thus the assertion follows. 
\end{proof}

Now we proceed to the proof of Theorem \ref{DisLLN}. 
We first show the almost sure convergence of \eqref{LLN}, and then show the $L^q$ convergence of \eqref{LLN} for $1 \le q < +\infty$.   

\begin{proof}[Proof of Theorem \ref{DisLLN} for the a.s. convergence of \eqref{LLN}]
Fix a vertex $o$ of $G$.
Since $G$ is vertex-transitive,  it holds that by \eqref{un-mean}, 
\[ \lim_{n \to \infty} \frac{\widetilde E^{p}[U_n]}{n} = E^{\Pp}\left[|C_{o}| P^{o}\left(T_{C_{o}} = +\infty\right)\right]. \]

We will check the assumptions of Liggett's subadditive ergodic theorem \cite{Li}.  
Recall \eqref{Umn-Def}.  
The subadditivity of $\{U_{m,n}\}_{m < n}$ is immediately seen.  
By the Markov property of $\{S_n\}$ and the translation invariance of Bernoulli percolation, 
$\{U_{m,m+n}\}_m$ is stationary. 
By the assumption, we have that 
\[ \widetilde E^{p}[U_{0,1}] \le 2E^{\Pp}[|C_o|] < +\infty.\]   

If the following lemma is shown, then, the almost sure convergence of \eqref{LLN} follows. 

\begin{Lem}\label{mix}
Assume that $p < p_T (G)$. 
$\{U_{nl, (n+1)l}\}_{n \ge 0}$ is strong mixing for any $l \ge 1$.   
Specifically, for any $k_{1}, k_{2} \ge 0$, 
\[ \lim_{m \to \infty} \widetilde P^{p}\left(\{U_{0, l} = k_{1}\} \cap \{U_{m, m+l} = k_{2}\}\right) 
= \widetilde P^{p}(U_{0, l} = k_{1}) \widetilde P^{p}(U_{0, l} = k_{2}). \]
\end{Lem}

For $A \subset G$ and $k \ge 0$,  the $k$-neighborhood of $A$ is defined by 
\[ A(k) := \{z \in G : \exists y \in A \textup{ such that } d(z,y) \le k\},\]

\begin{proof}
Informally speaking, we would like to show that $U_{0,l}$ and $U_{m,m+l}$ are ``asymptotically independent" as $m \to \infty$. 
We first decompose the event $\{U_{0,l} = k_1\}$ by possible positions of $\{S_i\}_{i=0}^{l}$, 
and then approximate each decomposed event by an event independent from $\{U_{m, m+l} = k_2\}$.  

Let 
\[ B(x_{0}, \dots, x_{l}) := \bigcap_{ 0 \le i \le l} \{S_{i} = x_{i}\} \cap \left\{\left|\bigcup_{i = 0}^{l} C_{x_{i}}\right| = k_{1} \right\}. \]
We decompose the event $\{U_{0, l} = k_{1}\}$ as follows : 
\[ \{U_{0, l} = k_{1}\} = \bigcup_{x_{0}, \dots, x_{l} \in G} B(x_{0}, \dots, x_{l}).\]

Since this union is disjoint, 
\[ \widetilde P^{p}(U_{0, l} = k_{1}) = \sum_{x_{0}, \dots, x_{l} \in G}  \widetilde P^{p}\left(B(x_{0}, \dots, x_{l})\right), \textup{ and, } \]
\[ \widetilde P^{p}\left(\{U_{0, l} = k_{1}\} \cap \{U_{m, m+l} = k_{2}\}\right)  = \sum_{x_{0}, \dots, x_{l} \in G}  \widetilde P^{p}\left(B(x_{0}, \dots, x_{l}) \cap \{U_{m, m+l} = k_{2}\}\right).  \]
 
Since the number of the possible candidates for $(x_{0}, \dots, x_{l})$ is finite, 
it suffices to show that
for a fixed sequence $(x_{0}, \dots, x_{l})$, 
$B(x_{0}, \dots, x_{l})$ and $\{U_{m, m+l} = k_{2}\}$ are asymptotically independent as $m \to \infty$,  
that is, 
\[  \lim_{m \to \infty} \widetilde P^{p}(B(x_{0}, \dots, x_{l}) \cap \{U_{m, m+l} = k_{2}\}) = \widetilde P^{p}(B(x_{0}, \dots, x_{l})) \widetilde P^{p}(U_{0, l} = k_{2}).  \]

We now consider events approximating $B(x_{0}, \dots, x_{l})$.  
Let $A_{l} := \{x_{0}, \dots, x_{l}\}$  
and \[ E_{m} := \left\{S_{m} \notin A_{l}(k_{1} + k_{2} + l + 2) \right\}.\]  

Since all infinite connected simple graphs have {\it at least linear growth}, 
by Woess \cite[Corollary 14.6]{W}, 
\[ P^{o}(S_{n} = o) \le Cn^{-1/2}, \] 
and hence,  
\[ \lim_{m \to \infty} P^{o}(E_{m}) = 1.\] 

Therefore it suffices to show that $B(x_{0}, \dots, x_{l}) \cap E_{m}$ and $\{U_{m, m+l} = k_{2}\}$ are independent, 
that is, 
\[ \widetilde P^{p}\left(B(x_{0}, \dots, x_{l}) \cap \{U_{m, m+l} = k_{2}\} \cap E_{m}\right) \] 
\begin{equation}\label{suff} 
=  \widetilde P^{p}\left(B(x_{0}, \dots, x_{l}) \cap E_{m}\right) \widetilde P^{p}(U_{0, l} = k_{2}).  
\end{equation}

It holds that   
\[ \widetilde P^{p}\left(B(x_{0}, \dots, x_{l}) \cap \{U_{m, m+l} = k_{2}\} \cap E_{m}\right)\]
\[ = \widetilde P^{p}\left(B(x_{0}, \dots, x_{l}) \cap E_{m} \cap \left\{\left|\bigcup_{z \in S[m, m+l]} C_{z}\right| = k_{2}\right\}\right)\]
\[ = \sum_{y_{0} \notin A_{l}(k_{1} + k_{2} + l + 2), y_{1}, \dots, y_{l}} 
\widetilde P^{p}\left(B(x_{0}, \dots, x_{l}) \cap \{S_{m+j} = y_{j}, 0 \le j \le l\} \cap \left\{\left|\bigcup_{j=0}^{l} C_{y_{j}}\right| = k_{2}\right\}\right)\]
\begin{align*} 
= \sum_{y_{0} \notin A_{l}(k_{1} + k_{2} + l + 2), y_{1}, \dots, y_{l}} 
P^{o}\left(\bigcap_{0 \le i \le l} \{S_{i} = x_{i}\} \cap  \bigcap_{0 \le j \le l} \{S_{m+j} = y_{j}\} \right) \\
\times \Pp\left(\left|\bigcup_{i=0}^{l} C_{x_{i}}\right| = k_{1}, \left|\bigcup_{j=0}^{l} C_{y_{j}}\right| = k_{2}\right). 
\end{align*}

Since $y_{0} \notin A_{l}(k_{1} + k_{2} + l + 2)$,
$\{|\cup_{i} C_{x_{i}}| = k_{1}\}$ and $\{|\cup_{j} C_{y_{j}}| = k_{2}\}$ are independent.
These events are completely determined by configurations in two finite boxes including $\{x_{i}\}_{i}$ and $\{y_{j}\}_{j}$, respectively. 
Therefore, 
\[ \widetilde P^{p}\left(B(x_{0}, \dots, x_{l}) \cap \{U_{m, m+l} = k_{2}\} \cap E_{m}\right) \]
\begin{align*} 
= \sum_{y_{0}; y_{1}, \dots, y_{l}} 
P^{o}\left(\bigcap_{0 \le i \le l} \{S_{i} = x_{i}\} \cap \{S_{m} = y_{0}\}\right) P^{y_{0}}\left(  \bigcap_{0 \le j \le l} \{S_{m+j} = y_{j}\} \right)\\
\times 
\Pp\left(\left|\bigcup_{i} C_{x_{i}}\right| = k_{1}\right)\Pp\left(\left|\bigcup_{j} C_{y_{j}}\right| = k_{2}\right) 
\end{align*}
\begin{align*} 
= \sum_{y_{0}} 
 P^{o}\left(\bigcap_{0 \le i \le l} \{S_{i} = x_{i}\} \cap \{S_{m} = y_{0}\}\right) \Pp\left(\left|\bigcup_{i} C_{x_{i}}\right| = k_{1}\right) \\
\times \left\{ \sum_{y_{1}, \dots, y_{l}} P^{y_{0}}\left(\bigcap_{ 0 \le j \le l } \{S_{j} = y_{j}\}\right)
\Pp\left(\left|\bigcup_{j} C_{y_{j}}\right| = k_{2}\right) \right\} 
\end{align*}
\[ = \sum_{y_{0} \notin A_{l}(k_{1} + k_{2} + l + 2)} 
 P^{o}\left(\bigcap_{0 \le i \le l} \{S_{i} = x_{i}\} \cap \{S_{m} = y_{0}\}\right) \Pp\left(\left|\bigcup_{i} C_{x_{i}}\right| = k_{1}\right) P^{y_{0}} \otimes \Pp(U_{0, l} = k_{2}) \]
\[ = \widetilde P^{p}(U_{0, l} = k_{2}) \widetilde P^{p}(B(x_{0}, \dots, x_{l}) \cap E_{m}).   \]
Thus we have \eqref{suff}.   
\end{proof}

\end{proof}

\begin{proof}[Proof of Theorem \ref{DisLLN} for the $L^q$-convergence of \eqref{LLN}]
It suffices to show that the following holds for any positive integer $q$:
\[ \sup_{n \ge 1} \widetilde E^p \left[\left(\frac{U_n}{n}\right)^q \right] < +\infty. \] 

By \cite[Theorem 3]{AV}, 
for each $p \in [0, p_T (G))$, there is a sufficiently small positive $\theta$ such that  
\begin{equation}\label{AV3}
E^{p}[\exp(\theta |C_o|)] < +\infty 
\end{equation}
holds for sufficiently small $\theta > 0$. 
By using the H\"older inequality,  
\[ \widetilde E^p \left[ \left(\frac{U_n}{n}\right)^q \right] \le \frac{q!}{\theta^q}  \widetilde E^p \left[\exp\left(\theta\frac{U_n}{n}\right) \right] \le  \frac{q!}{\theta^q}  \widetilde E^{p} \left[\exp\left(\theta |C_0|\right)\right] < +\infty. \] 
Hence, for any $q$, 
$\{(U_n / n)^q\}_{n \ge 1}$ are uniformly integrable.  
The $L^q$-convergence of \eqref{LLN} follows from this and the a.s. convergence of \eqref{LLN}. 
\end{proof}

\begin{Rem}
We state a second order expansion of $\widetilde E^p [U_n]$, which corresponds to \cite[Theorem 1]{DE} and \cite[Theorem 3.1]{P66} for the case that $G = \Z^d$, $d \ge 3$.  
By \eqref{un-mean} and \eqref{cp}, it is easy to see that 
if $G$ is vertex-transitive and transient, then, for a vertex $o$ of $G$,
\begin{align}\label{hk}
(1-p)^{\Delta_G} \sum_{k=1}^{n} P^o (k < T_o < +\infty)  &\le \widetilde E^{p}[U_n] - c_p n  \notag\\
&\le E^{\Pp}\left[ |C_o|^2 \right] \sum_{k=1}^{n} \sup_{x,y \in G} P^x (k < T_y < +\infty). 
\end{align}

We give a proof of these inequalities.  
By \eqref{un-mean}, \eqref{cp} and the fact that $G$ is vertex-transitive,
\[ \widetilde E^{p}[U_n] - c_p n = \sum_{k = 1}^{n} E^{\Pp}\left[ |C_{y}| P^{y}\left(k < T_{C_{y}} < +\infty\right)\right] \]
\[ \ge (1-p)^{\Delta_G} \sum_{k=1}^{n} P^o (k < T_o < +\infty). \]
Furthermore, by using the fact that 
$$ P^{y}\left(k < T_{C_{y}} < +\infty\right) \le P^{y}\left(\exists z \in C_y, k < T_{z} < +\infty\right) \le |C_y| \sup_{z \in C_y} P^{y}\left( k < T_{z} < +\infty\right), $$
\[ \widetilde E^{p}[U_n] - c_p n \le \sum_{k = 1}^{n} E^{\Pp}\left[ |C_{y}|^2 \sup_{z \in C_y} P^{y}\left( k < T_{z} < +\infty\right)\right] \]
\[\le E^{\Pp}\left[ |C_o|^2 \right] \sum_{k=1}^{n} \sup_{x,y \in G} P^x (k < T_y < +\infty). \]
Thus we have \eqref{hk}. 
In Remark \ref{hk-alt}, we give an alternative proof of \eqref{hk}. 
The magnitude of growth of $\sum_{k=1}^{n} \sum_{i \ge k} P^x (S_i = x)$ as a function of $n$ 
is $O(1)$ if $d \ge 5$, $O(\log n)$ if $d = 4$, and $O(n^{1/2})$ if $d = 3$, respectively.      

Furthermore, it is known (see Spitzer  \cite[p342]{S76}, for example) that there is a positive constant $c_d$ such that 
\[ c_d k^{1-d/2} \le P^0 (k < T_0 < +\infty), \ \ k \ge 1.\]  
Therefore, 
if $G = \Z^d, d \ge 3$, then, there are two positive constants $c_d$ and $C_d$ such that for any $n \ge 1$, 
\[ c_d (1-p)^{\Delta_G}   \le \frac{\widetilde E^{p}[U_n] - c_p n}{\sum_{k =1}^{n} k^{1-d/2}} \le C_d  E^{\Pp}\left[ |C_o|^2 \right]. \]
\end{Rem}


\section{Boundary of the trace}

This section is devoted to stating some results concerning the inner boundary of the trace of random walk, which will be used in the following Sections.   
Theorem \ref{lim-rec} is shown. 
Okada \cite{Okad} and Asselah-Schapira \cite{AS1, AS2} investigated a law of large numbers, variances, central limit theorems and tail estimates for random walks on $\Z^d$. 
Results we state below are new, unless we refer to the above references. 

\begin{Def}[Inner boundary of the trace]\label{bd-Def}
Let $\mathcal{N}(z)$ be the set of neighborhoods of a vertex $z$ of $G$, that is, 
\[ \mathcal{N}(z) := \left\{y \in G : \{z,y\} \in E(G) \right\}. \]
Let $\partial R_n$ be the set of $x \in \{S_0, \dots, S_n\}$ such that $\mathcal{N}(x) \not\subset \{S_0, \dots, S_n\}$. 
Let $L_n$ be the number of elements of $\partial R_n$. 
\end{Def}

We have that for any $n$,  
\begin{equation}\label{bd-always} 
R_n \le U_n \le R_n + \sum_{x \in \partial R_n}  |C_x|.  
\end{equation} 

Let $\partial^{\textup{e}} R_n$  be the set of $x \notin \{S_0, \dots, S_n\}$ such that $\mathcal{N}(x) \cap \{S_0, \dots, S_n\} \ne \emptyset$.  
Then, it holds that 
\[ L_n \le \Delta_G |\partial^{\textup{e}} R_n|, \]
and 
\[ E^{\Pp}[U_n] - R_n \ge p \left|\partial^{\textup{e}} R_n\right| \ge \frac{p}{\Delta_G} L_n. \]

By this and \eqref{bd-always},   
\begin{equation}\label{boundary} 
\frac{p}{\Delta_G} E^{P^o}[L_n] \le \widetilde E^{p}[U_n] - E^{P^o}[R_n]  \le \sup_{x \in G} E^{\Pp}[|C_x|] E^{P^o}[L_n]. 
\end{equation}

\begin{Lem}\label{bd-rec}
If $G$ is recurrent and vertex-transitive, then, for any vertex $x$ of $G$, 
\[ \lim_{n \to \infty} \frac{E^{P^x}[L_n]}{E^{P^x}[R_n]} = 0. \]
\end{Lem}

\begin{proof}
Let $\{S_{n}^{\prime}\}_n$ be a simple random walk on $G$ which is independent from $\{S_n\}_n$, 
let $T_z^{\prime}$ be the first hitting time of $z$  by $\{S^{\prime}_n\}$ and let $P^{x,y}$ be the joint law of $\{S_n\}_n$ which starts at $x$ and $\{S^{\prime}_n\}_n$ which starts at $y$.    
Then, by using the fact that $G$ is vertex-transitive, it holds that 
\begin{equation}\label{mean-L}
E^{P^x}[L_n] = \sum_{k = 1}^{n} P^{x, x}\left(T_x > k, \exists y \in \mathcal{N}(x) \textup{ such that } T_y > k \text{ and } T^{\prime}_y > n-k\right). 
\end{equation} 
Let $\epsilon > 0$. 
Then, by noting that $G$ is recurrent and vertex-transitive, there is a large number $M$ such that 
\[ \max_{y \in \mathcal{N}(x)} P^{x,x} (T^{\prime}_y > M) \le \epsilon. \]
By this and \eqref{mean-L}, 
it holds that 
\[ E^{P^x}[L_n] \]
\[ \le M +  \sum_{k = 1}^{n-M} \sum_{y \in \mathcal{N}(x)} P^{x, x}\left(T_x > k, T_y > k, T^{\prime}_y > n-k\right) \]
\[ \le M + \epsilon \Delta_G  \sum_{k = 1}^{n-M} P^x (T_x > k). \]
Since $G$ is vertex-transitive, 
\[ E^{P^x}[R_n] = \sum_{k = 0}^{n} P^x (T_x > k). \]
Hence,
\[ \frac{E^{P^x}[L_n]}{E^{P^x}[R_n]} \le \frac{M}{E^{P^x}[R_n]} + \epsilon \Delta_G.  \]

By using the monotone convergence theorem and the assumption that $G$ is recurrent, 
\[ \lim_{n \to \infty} E^{P^x}[R_n] = \lim_{n \to \infty} \sum_{y \in V(G)} P^x (T_y \le n) = \sum_{y \in V(G)} P^x (T_y < +\infty) = +\infty. \]
Therefore, 
\[ \limsup_{n \to \infty} \frac{E^{P^x}[L_n]}{E^{P^x}[R_n]} \le \epsilon \Delta_G.   \]
Since $\epsilon$ is taken arbitrarily, 
the assertion follows. 
\end{proof}

By Lemma \ref{bd-rec} and  \eqref{boundary}, 
we have Theorem \ref{lim-rec}.

We have that 
\[ P^{x, x}\left(T_x > k, \exists y \in \mathcal{N}(x) \textup{ such that } T_y > k \text{ and } T^{\prime}_y > n-k\right) \] 
\[ \ge P^{x, x}\left(T_x > k, \exists y \in \mathcal{N}(x) \textup{ such that } T_y = +\infty \text{ and } T^{\prime}_y  = +\infty \right). \]   
By this and \eqref{mean-L},  we have that if $G$   is vertex-transitive, 
\begin{equation}\label{joint} 
\lim_{n \to \infty} \frac{E^{P^o}[L_n]}{n} \ge P^{o, o}\left(\{T_o = +\infty\} \cap \bigcup_{y \in \mathcal{N}(o)} \left\{T_y = T_y^{\prime} = +\infty \right\} \right), \ o \in V(G). 
\end{equation}

\begin{Rem}
By \cite[Theorems 5.12 and 5.13]{W}, 
any vertex-transitive recurrent graph is a $d$-dimensional generalized lattice, $d = 1 \textup{ or } 2$, 
that is, a graph whose automorphism group contains the free group $\Z^d$ as a quasi-transitive subgroup.     
\end{Rem}

\begin{Lem}\label{lem-okada}
Let $G = \Z^2$. Then, \\
(i) (\cite[Theorem 2.4]{Okad}) 
There is a constant $c \in [\pi^2 /2, 2\pi^2]$ such that 
\begin{equation}\label{Okada24}
\lim_{n \to \infty} \frac{(\log n)^2}{n} E^{P^0}[L_n] = c. 
\end{equation}
(ii) 
\begin{equation}\label{L-2nd} 
\limsup_{n \to \infty} \frac{(\log n)^4}{n^2} E^{P^0}[L_n^2] <+\infty. 
\end{equation} 
(iii) 
\begin{equation}\label{Okad-as} 
\lim_{n \to \infty} \frac{L_n}{R_n} = 0, \ \textup{ $P^0$-a.s.} 
\end{equation}
\end{Lem}

\begin{proof} 
See \cite{Okad} for the proof of (i). 

(ii) Let $u_n = \exp(n^{2/3})$ and $a_n = u_n^{1/4}$. 
If $\exp(n^{2/3})$ or $u_n^{1/4}$ is not an integer, we take the integer part of it. 
Let
\[ V_n := \left|\left\{x \in S[0,u_n] : \mathcal{N}(x) \not\subset S[T_x, T_x + a_n]\right\}\right|. \] 
Then
\begin{equation}\label{LV}
\max_{k \in [u_{n-1}, u_n]} L_k \le V_n + a_n. 
\end{equation}

For $k \in (a_n, u_n - a_n)$, let 
\[ A_k := \left\{S_k \notin S[k- a_{n} -1, k-1], \mathcal{N}(S_k) \not\subset S[T_{S_k}, T_{S_k} + a_n] \right\}.  \] 
Then if $|k_1 - k_2| > a_n$ then 
$A_{k_1}$ and $A_{k_2}$ are independent.  
Therefore, 
\begin{align*} 
E^{P^0}[V_n^2] &\le \left|\{(k_1, k_2) : |k_1 - k_2| \le a_n \textup{ or } k_1 \notin (a_n, u_n - a_n) \textup{ or } k_2 \notin (a_n, u_n - a_n) \}\right| \\
&+ \sum_{k_i \in  (a_n, u_n - a_n)} P^0 (A_{k_1})P^0 (A_{k_2}).  
\end{align*}
By Kesten and Spitzer\footnote{In \cite{KS} it is stated that \cite[Theorem 4a]{KS} holds for aperiodic random walk, but the definition of aperiodicity in \cite{KS} is different from the usual definition of it. The usual definition of aperiodicity is that the infimum of $n$ such that $P^x(S_n = x)$ is positive. We can apply this result to the simple random walk.} \cite[Theorem 4a]{KS},  
\[ P^0 (A_k) \le P^0 (T_0 > a_n) P^0 \left(\mathcal{N}(0) \not\subset S[0, a_n]\right) \le O\left((\log a_n)^{-2}\right). \] 
By using this and 
\[ \left|\left\{(k_1, k_2) : |k_1 - k_2| \le a_n \textup{ or } k_1 \notin (a_n, u_n - a_n) \textup{ or } k_2 \notin (a_n, u_n - a_n) \right\}\right| = O(a_n u_n), \] 
it holds that 
\begin{equation}\label{V2-order}
E^{P^0}\left[V_n^2\right] = O\left(\frac{u_n^2}{(\log u_n)^4}\right). 
\end{equation} 

Recall $u_n = \exp(n^{2/3})$, and $a_n = o\left(u_n / (\log u_n)^{2} \right)$.  
Now \eqref{L-2nd} follows from  \eqref{LV} and  \eqref{V2-order}. 

(iii) We now show \eqref{Okad-as}. 
Since $a_n = o\left( u_n / \log u_n \right)$ and 
\[ \lim_{n \to \infty} \frac{u_n / \log u_n}{u_{n-1} / \log u_{n-1}} = 1, \]
it suffices to show that  
\begin{equation}\label{Vn-conv}
\lim_{n \to 0} \frac{\log u_n}{u_n} V_n = 0, \textup{ $P^0$-a.s.}  
\end{equation}

By \eqref{V2-order},  
\[ P^0 \left(V_n > \frac{u_n}{\log u_n \log\log u_n} \right) \le \left(\frac{\log u_n \log\log u_n}{u_n}\right)^2 E[V_n^2] = O\left(\left(\frac{\log\log u_n}{\log u_n}\right)^2\right). \] 
By using the Borel-Cantelli lemma, 
we have \eqref{Vn-conv}. 
\end{proof}


\section{Properties of $c_p$}

This section is devoted to investigating properties of the limit $c_p$ as a function of $p$.  

\begin{proof}[Proof of Theorem \ref{cp-property} (i)] 
Fix a vertex $o$ of $G$. 
The following proof for the analyticity of $c_p$ is almost identical to the proof of Grimmett \cite[Theorem 6.108]{Gr}, so we give a sketch only. 
Fix a vertex $o$ of $G$.    
We have 
\begin{equation*}
c_p = \sum_{n \ge 1} n \sum_{A \subset G, o \in A,  |A| = n} P^o (T_A = +\infty) \Pp(C_o = A). 
\end{equation*}

Let  $a_{n, m, b}$ be the number of $A \subset G$ such that $o \in A,  |A| = n, \textup{ and } \Pp(C_o = A) = p^m (1-p)^b$. 
Let 
\[ a^{\prime}_{n, m, b} := \sum_{A \subset G, o \in A,  |A| = n, \Pp(C_o = A) = p^m(1-p)^b} P^o (T_A = +\infty). \]
Then it holds that $a^{\prime}_{n, m, b} \le a_{n, m, b}$ and 
\[ c_p = \sum_{n \ge 1, m, b \ge 0} n a^{\prime}_{n, m, b} p^m (1-p)^b. \]
If $a^{\prime}_{n, m, b} > 0$ then $m \le \Delta_G n$ and $b \le \Delta_G n$.  
By replacing $a_{n, m, b}$ with $a^{\prime}_{n, m, b}$ in the proof of \cite[Theorem 6.108]{Gr}, we have the analyticity of $c_p$ if $p$ is small. 

Let $$K(z) := \sum_{n \ge 1} n \sum_{m, b = 0}^{\Delta_G n} a^{\prime}_{n, m, b} z^m (1-z)^b.$$ 
Let $0 < \alpha < \beta < p_T(G)$. 
We will show that $K$ is uniformly convergent on a domain in the complex plane containing $[\alpha, \beta]$ in its interior. 
Let $p \in [\alpha, \beta]$. 
Let $\delta > 0$.  
Assume that  $|z - p| < \delta$. 
Then, 
\[ \left|n \sum_{m, b = 0}^{\Delta_G n} a^{\prime}_{n, m, b} z^m (1-z)^b\right| \le n \sum_{m, b = 0}^{\Delta_G n} a_{n, m, b} (p+\delta)^m (1-p+\delta)^b\]
\[ \le n \left(\frac{p+\delta}{p} \cdot \frac{1-p+\delta}{1-p}\right)^{\Delta_G n} \sum_{m, b = 0}^{\Delta_G n} a_{n, m, b} p^m (1-p)^b \]
\[ = n \left(\frac{p+\delta}{p} \cdot \frac{1-p+\delta}{1-p}\right)^{\Delta_G n} \Pp(|C| = n). \]

By Antunovi\'c-Veseli\'c \cite[Theorem 3]{AV}, 
we have the exponential decay of sizes of clusters of subcritical percolations on $G$.  
That is, there exist two positive constants $c_1(\beta), c_2(\beta) > 0$ such that 
\[ \Pp(|C| = n) \le \mathbb{P}_{\beta} (|C| \ge n) \le c_1(\beta) \exp(-c_2(\beta) n). \]
If we take sufficiently small $\delta > 0$, then, 
\[ \lim_{n \to \infty} n \left(\frac{p+\delta}{p} \cdot \frac{1-p+\delta}{1-p}\right)^{\Delta_G n} c_1(\beta) \exp(-c_2(\beta) n) = 0.\] 
Hence
$K$ is analytic on a domain in the complex plane containing $[\alpha, \beta]$ in its interior, and 
the analyticity of $c_p$ on $p \in [0, p_{T}(G))$ now holds.        
\end{proof}

\begin{Def}
We define the {\it capacity} for subsets of $V(G)$ in terms of the {\it effective resistance}.  
If $A$ is finite, then, we let 
\begin{equation}\label{ca-re}
\ca(A) := \lim_{n \to \infty} \left( R_{\textup{eff}}(A, B(x, n)^c) \right)^{-1}. 
\end{equation} 
\end{Def}

Then, 
\begin{equation}\label{mon-cap}
\ca(A) \le \ca(B), \ A \subset B \subset G.  
\end{equation}

By the argument following Kumagai \cite[Theorem 2.2.5]{Ku}, 
\begin{equation}\label{cap-cap} 
\ca(A) = \sum_{x \in A} P^x (T_A = +\infty), \ A \subset G. 
\end{equation}

\begin{Lem}
If $V(G)$ has a structure of group and any left multiplication induces a graph homomorphism, 
then, 
\begin{equation}\label{cp-cap}
c_p = E^{\Pp} \left[\ca(C_x)\right], \ x \in V(G).   
\end{equation}
\end{Lem}

We write $x \leftrightarrow y$ if $x$ and $y$ are connected by an open path. 

\begin{proof}
Let $o$ be the unit element of $V(G)$ as group. 
$-x$ denotes the inverse element of an element $x$ as group. 
By \eqref{cap-cap}, 
it holds that 
\[ E^{\Pp}\left[\ca(C_o)\right]  
= \sum_{x \in G} E^{\Pp}\left[ P^{x}(T_{C_o} = +\infty), \ o \leftrightarrow x\right] \]
\[ = \sum_{x \in G} E^{\Pp}\left[ P^{x}(T_{C_x} = +\infty), \ o \leftrightarrow x\right]\]
\[ = \sum_{x \in G} E^{\Pp}\left[ P^{o}(T_{C_o} = +\infty), \ o \leftrightarrow -x\right] 
= c_{p}. \]
\end{proof}

\begin{proof}[Proof of Theorem \ref{cp-property} (ii)] 
Let $p_1 < p_2$. 
Let $p_3 > 0$ such that 
\begin{equation}\label{def-p3} 
p_1 + p_3 - p_1 p_3 = p_2. 
\end{equation}
We regard the percolation with parameter $p_2$ as the independent union of percolation with parameter $p_1$ and percolation with parameter $p_3$.  

Let $C_o^i$, $i = 1,2,3$, be the open clusters containing $o$.  
Then, 
\begin{equation*}\label{123} 
E^{\mathbb{P}_{p_2}} \left[ \ca(C_o^2) \right] = E^{\mathbb{P}_{p_1}}  \left[ E^{\mathbb{P}_{p_3}}  \left[\ca(C_o^1 \cup C_o^3)\right]\right].  
\end{equation*}

By this, \eqref{mon-cap}, \eqref{cp-cap}, and $\mathbb{P}_{p_1}(C_o^1 = \{o\}) = (1- p_1)^{\Delta_G}$,  
we have  that 
\begin{align}\label{p2-p1}
 c_{p_2}-  c_{p_1} &=  E^{\mathbb{P}_{p_1}}  \left[ E^{\mathbb{P}_{p_3}}  \left[\ca(C_o^1 \cup C_o^3)\right] - \ca(C_o^1)\right] \notag \\
&\ge  E^{\mathbb{P}_{p_1}}  \left[ E^{\mathbb{P}_{p_3}}  \left[\ca(C_o^1 \cup C_o^3)\right] - \ca(C_o^1),\  C_o^1 = \{o\} \right] \notag \\ 
&\ge (1- p_1)^{\Delta_G} (c_{p_3} - c_0). 
\end{align}  

On the other hand, by \eqref{boundary}, 
\[ \frac{E^{\widetilde P^{o,p}}[U_n] - E^{P^o} [R_n] }{n} \ge \frac{\delta}{\Delta_G} \frac{E^{P^{o}}[L_n]}{n}. \]
By \eqref{joint} and Theorem \ref{DisLLN},
\[ c_{\delta} - c_{0} = \lim_{n \to \infty} \frac{\delta}{\Delta_G} P^{o,o} \left(\{T_o = +\infty\} \cap \bigcup_{y \in \mathcal{N}(o)} \{T_y = T_y^{\prime} = +\infty \} \right). \]
By the assumption that $G$ is transient,  
\begin{align*}  
\liminf_{\delta \to 0} \frac{c_{\delta} -c_{0}}{\delta} 
\ge \frac{1}{\Delta_G} P^{o,o} \left(\{T_o = +\infty\} \cap \bigcup_{y \in \mathcal{N}(o)} \{T_y = T_y^{\prime} = +\infty \} \right) > 0. 
\end{align*} 

Now \eqref{cp-modc} follows from \eqref{p2-p1} and \eqref{def-p3}. 
\end{proof}

\begin{proof}[Proof of Theorem \ref{cp-property} (iii)] 
This is obtained by a combination of two results. 

By the proof of Lawler \cite[Proposition 2.5.1]{La}, 
there exists a constant $c_d$ such that for every non-empty subset $A$ of $\mathbb{Z}^d$, 
\[ \ca(A) \ge c_d |A|^{1 - 2/d}. \]

By \eqref{cp-cap} and this, 
\[ c_p \ge c_d E^{\Pp}\left[\left|C_0\right|^{1 - 2/d}\right]. \]

By Fitzner-van der Hofstad \cite[Corollary 1.3 and (1.8)]{FH17}, if $d \ge 11$, 
\[ \mathbb{P}_{p_c(\Z^d)}(|C_0| > n) \simeq n^{-1/2}. \]

\[ E^{\Pp}\left[\left|C_0\right|^{1 - 2/d}\right] = \frac{d-2}{d} \sum_{n \ge 1} n^{-2/d}  \mathbb{P}_{p}(|C_0| > n). \]
By the monotone convergence theorem, 
\[ \lim_{p \to p_c} E^{\Pp}\left[\left| C_0 \right|^{1 - 2/d}\right] = \frac{d-2}{d} \sum_{n \ge 1} n^{-2/d}  \mathbb{P}_{p_c}(|C_0| > n).\]
Since $n^{-2/d}  \mathbb{P}_{p_c}(|C_0| > n) \simeq  n^{-(1/2 + 2/d)}$ and $1/2 + 2/d < 1$,  
\[ \sum_{n \ge 1} n^{-2/d}  \mathbb{P}_{p_c}(|C_0| > n) = +\infty.\] 

Thus we have the assertion. 
\end{proof} 

\begin{Rem}
Two random variables $-|C_o|$ and $P^o (T_{C_o} = +\infty)$ are both decreasing random variables under $\Pp$. 
Then, by the FKG inequality, 
\[ E^{\Pp}\left[(-|C_o|) P^o (T_{C_o} = +\infty)\right] \ge E^{\Pp}\left[(-|C_o|) \right]  \widetilde P^{o, p}(T_{C_o} = +\infty).  \]
Hence we have the following upper bound for $c_p$:   
\[ c_p \le E^{\Pp}[|C_o|] \widetilde P^{o, p} (T_{C_o} = +\infty) < E^{\Pp}[|C_o|] P^o(T_o = +\infty). \]  
\end{Rem}


\section{Finite modification and fluctuation}

In this section, Theorems \ref{fm} and \ref{fluc}  are shown. 
We first deal with finite modifications of graphs.  

Let the Hammmersley critical probability 
\[ p_H (G) := \inf\{p \in [0,1] : \Pp(|C_x| = +\infty) > 0\}, \ x \in V(G). \]
This value does not depend on the choice of $x$. 

\begin{Lem}\label{fm-cri-pb}
(i) There are two positive constants $C$ and $c$ such that for any vertex $x$ of $G^{\prime}$ and $n \ge 1$,  
\[ {\Pp}^{G^{\prime}} \left( |C_x| > n \right) \le C \exp(-cn). \]
(ii) \[ p_H (G) = p_T (G) = p_H (G^{\prime}) = p_T (G^{\prime}). \]
\end{Lem}

\begin{proof}
Let $\phi : G \setminus D \to  G^{\prime} \setminus D^{\prime}$ be a graph isomorphism. 

(i) There is nothing to show if $p = 0$. So we assume that $p > 0$.   
Let 
\[ E(G \setminus D) := \{\{x,y\} \in E(G) : x, y \in G \setminus D\} \]
and 
\[ E(G^{\prime} \setminus D^{\prime}) := \{\{x,y\} \in E(G^{\prime}) : x, y \in G^{\prime} \setminus D^{\prime}\}. \]
Now we can decompose $\{|C_x| > n \}$ as follows: 
\[ \{|C_x| > n \} = \bigcup_{\omega \in \{0,1\}^{E(G^{\prime}) \setminus E(G^{\prime} \setminus D^{\prime})}} \{\omega\} \times A(\omega), \]
where $A(\omega) \subset \{0,1\}^{E(G^{\prime} \setminus D^{\prime})}$. 
Hence, 
\[ {\Pp}^{G^{\prime}} \left( |C_x| > n \right) \le (\max\{p, 1-p\})^{|E(G^{\prime}) \setminus E(G^{\prime} \setminus D^{\prime})|} \sum_{\omega \in \{0,1\}^{E(G^{\prime}) \setminus E(G^{\prime} \setminus D^{\prime})}} {\Pp}^{G^{\prime} \setminus D^{\prime}}(A(\omega)) \]
where we denote the Bernoulli measure with parameter $p$ on $\{0,1\}^{E(G^{\prime} \setminus D^{\prime})}$ by ${\Pp}^{G^{\prime} \setminus D^{\prime}}$.

Let $O_1$ be the event that all edges of $E(G) \setminus E(G \setminus D)$ are open. 
By identifying $G \setminus D$ and $G^{\prime} \setminus D^{\prime}$, 
\[ {\Pp}^{G^{\prime} \setminus D^{\prime}}(A(\omega)) = \frac{{\Pp}^{G}(O_1 \times A(\omega))}{p^{|E(G) \setminus E(G \setminus D)|}} \] 
Fix a vertex $z$ of $D$. 
Then, 
\[ O_1 \times A(\omega) \subset \{ |C_z| > n - |D^{\prime}| \} \]
Hence,
\[  {\Pp}^{G^{\prime}} \left( |C_x| > n \right) \le \frac{(\max\{p, 1-p\})^{|E(G^{\prime}) \setminus E(G^{\prime} \setminus D^{\prime})|}}{p^{|E(G) \setminus E(G \setminus D)|}} {\Pp}^{G}\left(|C_z| > n - |D^{\prime}|\right). \]
By this and \cite{AV},  we have the assertion. 

(ii) 
By (i), $p_H (G) \le p_H (G^{\prime}) = p_T (G^{\prime})$.      
By \cite{AV}, we also have $p_H (G) = p_T (G)$.     
Assume $p > p_H (G) = p_T (G)$. 
Then, by using the fact that  the exterior boundary of $D^{\prime}$ is finite and classifying any infinite self-avoiding paths of $G$ by the last exit point from $D$, 
we have that for some $x$ in the exterior boundary of $D^{\prime}$, 
\[ {\Pp}^{G^{\prime}} (|C_x| = +\infty) \ge {\Pp}^{G}  \left(\textup{there is an infinite path from $\phi^{-1}(x)$ in $G \setminus D$}\right) > 0. \]
Hence $p > p_H (G^{\prime}) = p_T (G^{\prime})$. 
Since $p$ is taken arbitrarily, $p_H (G) \ge p_H (G^{\prime})$.  
\end{proof} 

\begin{proof}[Proof of Theorem \ref{fm}]   
Let $G^{\prime}$ be a finite modification of $G$.  
Let $\phi : G \setminus D \to  G^{\prime} \setminus D^{\prime}$ be a graph isomorphism. 
In this proof, constants (denoted by $C$, $c$ etc) depend only on $G$ and $G^{\prime}$.   

Let $o$ be a vertex of $G^{\prime}$.  
Here $D(k)$ and $D^{\prime}(k)$ denotes the $k$-neighborhoods of $D$ in $G$, and $D^{\prime}$ in $G^{\prime}$, respectively.    

(i) First we give a rough idea of proof. 
If the random walk exits a large ball containing $D^{\prime}$, 
then, with high probability it does not return $D^{\prime}$ again and the behavior of the random walk is identical with the behavior of the simple random walk on $G$.   

Let 
\[ T^{(n)}_{D^{\prime}} := \inf \left\{i > T_{B_{G^{\prime}} (o, n)^c} : S_i \in D^{\prime} \right\}. \]

Fix $m > 4N_0$.  
Let $x$ be a vertex of $G^{\prime}$ such that $d_{G^{\prime}}(x, o) > 2m + \textup{diam}(D^{\prime})$. 
Then, by Theorem \ref{DisLLN}, 
\[ \lim_{n \to \infty} \frac{U_n}{n} = c_{G,p}, \ P_{G}^{\phi^{-1}(x)} \otimes {\Pp}^{G} \textup{-a.s.}  \]
This implies that 
\[ \lim_{n \to \infty} \frac{U_n}{n} = c_{G,p}, \ P_{G}^{\phi^{-1}(x)} \otimes {\Pp}^{G} \textup{-a.s. on } \{ T_{D} = +\infty\} \times \{D \not\leftrightarrow G \setminus D(m) \}.  \]

By this and the definition of $G$, 
\[  \lim_{n \to \infty} \frac{U_n}{n} = c_{G,p}, \ P_{G^{\prime}}^{x} \otimes {\Pp}^{G^{\prime}} 
\textup{-a.s. on } \{ T_{D^{\prime}} = +\infty\} \times \{D^{\prime} \not\leftrightarrow G^{\prime} \setminus  D^{\prime}(m) \}.  \]

By this and the strong Markov property, 
\[ \lim_{n \to \infty} \frac{U_{T_{B_{G^{\prime}}(o, 3m)^c}, n}}{ n - T_{B_{G^{\prime}}(o, 3m)^c}} = c_{G,p},  P_{G^{\prime}}^{o} \otimes {\Pp}^{G^{\prime}} \textup{-a.s. on } \{T^{(3m)}_{D^{\prime}} = +\infty\} \times \{D^{\prime} \not\leftrightarrow G^{\prime} \setminus  D^{\prime}(m) \}.  \]
Here and henceforth we let 
\[  \frac{U_{T_{B_{G^{\prime}}(o, 3m)^c}, n}}{ n - T_{B_{G^{\prime}}(o, 3m)^c}} := 0, \textup{ if } n \le T_{B_{G^{\prime}}(o, 3m)^c}.  \]

It holds that 
\[ U_n = U_{0,m} + U_{m,n} - \left| \left( \bigcup_{i \in [0,m]} C_{S_i}\right) \cap \left( \bigcup_{i \in [m, n]} C_{S_i} \right) \right|, \ \ 0 \le m \le n. \]
By using the transience of $G^{\prime}$, 
\[ {\Pp}^{G^{\prime}} \left(T_{B_{G^{\prime}}(o, 3m)^c} < +\infty\right) = 1. \]
Therefore, for each $m$, 
\[ \lim_{n \to \infty} \left|\frac{U_n}{n} -  \frac{U_{T_{B_{G^{\prime}}(o, 3m)^c}, n}}{ n - T_{B_{G^{\prime}}(o, 3m)^c}}\right| = 0, \ P_{G^{\prime}}^{o} \otimes {\Pp}^{G^{\prime}} \textup{-a.s.}\]
Therefore, 
\[ \lim_{n \to \infty}  \frac{U_{n}}{n} = c_{G,p}, \ P_{G^{\prime}}^{o} \otimes {\Pp}^{G^{\prime}} \textup{-a.s. on } \{T^{(3m)}_{D^{\prime}} = +\infty\} \times \{D^{\prime} \not\leftrightarrow G^{\prime} \setminus  D^{\prime}(m) \}.  \]

By using the transience of $G^{\prime}$,  we have that 
\[ \lim_{m \to \infty} P_{G^{\prime}}^{o} \left(T^{(3m)}_{D^{\prime}} = +\infty\right) = 1. \]

By noting that $p < p_T (G) = p_T (G^{\prime})$ and the finiteness of $D$ and $D^{\prime}$,  
\[ \lim_{m \to \infty} {\Pp}^{G^{\prime}} (D^{\prime} \not\leftrightarrow G^{\prime} \setminus D^{\prime}(m)) =  \lim_{m \to \infty} {\Pp}^{G} (D \not\leftrightarrow G \setminus D(m)) = 1. \]
Since the event $\{ T^{(3m)}_{D^{\prime}} = +\infty\} \times \{D^{\prime} \not\leftrightarrow  G^{\prime} \setminus D^{\prime}(m)\}$ is increasing with respect to $m$, we have \eqref{fm-as}.  

Now we show \eqref{Lv1-conv}. 

\begin{Lem}\label{uc}  
Assume that $G^{\prime}$ satisfies 
\begin{equation}\label{d2} 
\lim_{k \to \infty} \sup_{x, y \in G^{\prime}} P_{G^{\prime}}^x (k < T_y < +\infty) = 0. 
\end{equation}
and $\{|C_x| : x \in G^{\prime}\}$ are uniformly integrable with respect to $\Pp$. 
Then, 
\[ \lim_{n \to \infty} \sup_{x \in G^{\prime}} E^{\Pp}_{G^{\prime}} \left[ |C_x| P^x_{G^{\prime}} (n < T_{C_x} < +\infty) \right] = 0. \]
\end{Lem}

The assumption of uniform integrability above is satisfied due to Lemma \ref{fm-cri-pb}. 
Since $G$ is vertex-transitive and satisfies $V(d)$ for some $d > 2$, by noting \cite[Corollary 14.5]{W}, 
the heat kernel of $G$ satisfies the Nash inequality of order $d/2 > 1$. 
By using the fact that $G^{\prime}$ is roughly isometric to $G$ and the stability of the Nash inequality under rough isometries, 
the heat kernel of $G^{\prime}$ satisfies the Nash inequality of order $d/2$.  
Hence \eqref{d2} holds. 

\begin{proof} 
It follows that for each $x \in G^{\prime}$,  
\[ E_{G^{\prime}}^{\Pp}\left[ |C_x| P_{G^{\prime}}^x (n < T_{C_x} < +\infty) \right] \]
\[ \le \sup_{x \in G^{\prime}} E_{G^{\prime}}^{\Pp}\left[ |C_x|^2 \right] \sup_{x \in G^{\prime}} E^{\Pp}_{G^{\prime}} \left[ |C_x| P^x_{G^{\prime}} (n < T_{C_x} < +\infty) \right]\] 
Now the assertion follows from this and the assumption of uniform integrability. 
\end{proof}

Let $\epsilon > 0$. 
Then, by Lemma \ref{uc}, there is $m_0$ such that for some (or equivalently any) $x \in G$, 
\begin{equation}\label{infty-infty}  
E^{\Pp}_{G} \left[ |C_x| P_{G}^x (m_0 < T_{C_x} < +\infty) \right]  
+ \sup_{y \in G^{\prime}} E^{\Pp}_{G^{\prime}} \left[ |C_y| P_{G^{\prime}}^y (m_0 < T_{C_y} < +\infty) \right]  \le \epsilon,   
\end{equation}  
and, 
\begin{equation}\label{infty-infty-2}  
E^{\Pp}_{G} \left[ |C_x|, |C_x| \ge m_0 /2 \right] + \sup_{y \in G^{\prime}} E^{\Pp}_{G^{\prime}} \left[ |C_y|, |C_y| \ge m_0 /2 \right] \le \epsilon.   
\end{equation} 
Furthermore, the structure of $G^{\prime} \setminus B(o, m_0)$ is the same as a subgraph of $G$.   

By \cite[Proposition 4.3.2]{Ku},  
there is $n_0 > 2m_0$ such that for any $k \ge n_0$
\[ P^o_{G^{\prime}} \left(S_k \in B_{G^{\prime}}(o, 2m_0)\right)  \le \epsilon.  \]

Then, for any $n > n_0$,  
\[ \left| \widetilde E^{o,p}_{G^{\prime}} [U_n]  - \sum_{k=0}^{n} E^{\Pp}_{G} \left[ |C_x| P_{G}^x (k < T_{C_x}) \right]\right| \]
\[ \le \sum_{k=0}^{n} \sum_y P^o_{G^{\prime}} (S_{k} = y)  \left| E_{G^{\prime}}^{\Pp} [|C_y| P^y_{G^{\prime}} (T_{C_y} > n-k)] - E^{\Pp}_{G} \left[ |C_x| P_{G}^x (T_{C_x} > n-k) \right]\right|  \]
\[ \le (n_0 +  (n - n_0) \epsilon + m_0) \left( \sup_{y \in G^{\prime}} E^{\Pp}_{G^{\prime}}  \left[ |C_y| \right] +  E^{\Pp}_{G}  \left[ |C_x| \right] \right)  \]
\begin{equation}\label{3.5.1} 
+ \sum_{k = n_0}^{n - m_0} \sum_{y \notin B_{G^{\prime}}(o, 2m_0)} P^o_{G^{\prime}} (S_n = y) 
\left|E^{\Pp}_{G^{\prime}} \left[ |C_y| P_{G^{\prime}}^y (n-k < T_{C_y}) \right] - E^{\Pp}_{G} \left[ |C_x| P_{G}^x (n-k < T_{C_x}) \right]\right|.   
\end{equation}

By \eqref{infty-infty}, it follows that for any $y \notin B(o, 2m_0)$ and $l > m_0$,  
\[ \left| E^{\Pp}_{G^{\prime}} \left[ |C_y| P_{G^{\prime}}^y (l < T_{C_y}) \right] - E^{\Pp}_{G} \left[ |C_x| P_{G}^x (l < T_{C_x}) \right] \right| \]
\begin{equation}\label{3.5.2}
\le 2\epsilon + \left| E^{\Pp}_{G^{\prime}} \left[ |C_y| P_{G^{\prime}}^y (m_0 < T_{C_y}) \right] - E^{\Pp}_{G} \left[ |C_x| P_{G}^x (m_0 < T_{C_x}) \right] \right|. 
\end{equation} 

By the assumption of finite modification, 
If $y \notin B_{G^{\prime}}(o, 2m_0)$ and a connected subset $A$ such that $y \in A$ and $|A| \le m_0 /2$, 
then,  
\[ P_{G^{\prime}}^y (m_0 < T_{A}) = P_{G}^y (m_0 < T_{A}).  \]
Here we have identified vertices on $G^{\prime} \setminus B_{G^{\prime}}(o, 2m_0)$ and $G$. 
Hence, 
\[ E^{\Pp}_{G^{\prime}} \left[ |C_y| P_{G}^y (m_0 < T_{C_y}),  |C_y| \le m_0 /2 \right] = E^{\Pp}_{G} \left[ |C_x| P_{G}^y (m_0 < T_{C_x}),  |C_x| \le m_0 /2 \right].  \]

By this and \eqref{infty-infty-2}, 
it holds that 
\begin{equation}\label{3.5.3} 
\left| E^{\Pp}_{G^{\prime}} \left[ |C_y| P_{G}^y (m_0 < T_{C_y}) \right] - E^{\Pp}_{G} \left[ |C_x| P_{G}^x (m_0 < T_{C_x}) \right] \right| \le 2\epsilon. 
\end{equation}

By \eqref{3.5.1}, \eqref{3.5.2} and \eqref{3.5.3},  for some constant $C$, 
\[ \limsup_{n \to \infty} \frac{1}{n} \left| \widetilde E^{o,p}_{G^{\prime}} [U_n]  - \sum_{k=0}^{n-1} E^{\Pp}_{G} \left[ |C_x| P_{G}^x (k < T_{C_x}) \right]\right| 
\le C\epsilon. \]
Since $\epsilon > 0$ has been taken arbitrarily, 
\[ \limsup_{n \to \infty} \frac{1}{n} \left| \widetilde E^{o,p}_{G^{\prime}} [U_n]  - \sum_{k=0}^{n-1} E^{\Pp}_{G} \left[ |C_x| P_{G}^x (k < T_{C_x}) \right] \right| = 0.\]
Since \[ \lim_{k \to \infty} E^{\Pp}_{G} \left[ |C_x| P_{G}^x (k < T_{C_x} < +\infty) \right] = 0,\]  
we have \eqref{Lv1-conv}. 

Now we recall the following result by Br\'ezis-Lieb \cite{BL83}.

\begin{Thm}
Let $(X, \mathcal{B}, \mu)$ be a measure space and $(f_n)_{n \ge 1}, f$ be $L^p$-integrable functions on $X$ for some $p \ge 1$.  
Assume that $f_n \to f$ $\mu$-a.e. and $\| f_n \|_p \to \| f \|_p$. Then, 
$\|f_n - f\|_p \to 0$.     
\end{Thm}

\eqref{Lv1-conv-L1} follows from this, \eqref{fm-as} and  \eqref{Lv1-conv}. \\

(ii) Let $\theta = d/2 > 2$. 
Assume that $D^{\prime}$ is contained in $B_{G^{\prime}}(o, N_0)$. 

By \eqref{last-exit}, 
\[ \widetilde E^{o,p}_{G^{\prime}}[U_n] =  \sum_{k \le n} \sum_y P_{G^{\prime}}^o (S_{k} = y) E_{G^{\prime}}^{\Pp} \left[|C_y| P_{G^{\prime}}^y(T_{C_y} > n-k)\right]. \]

Hence, 
\[ \widetilde E_{G^{\prime}}^{o,p}[U_n - U_{n-1}] \]
\[= \widetilde E_{G^{\prime}}^{o,p}[|C_{S_n}|] - \sum_{y \in G^{\prime}} \sum_{k \le n-1}  P^o_{G^{\prime}}(S_k = y) E^{\Pp}_{G^{\prime}} \left[|C_y| P^y (T_{C_y} = n-k)\right].  \]

We compare $\widetilde E_{G^{\prime}}^{o,p}[U_n - U_{n-1}]$ with $E^{\Pp}_{G}\left[|C_x| P^x (T_{C_x} = +\infty)\right]$. 
Our strategy is to compare $\widetilde E_{G^{\prime}}^{o,p}[|C_{S_n}|]$ with $E^{\Pp}_{G}[|C_x| ]$ first and compare 
$\sum_{y \in G^{\prime}} \sum_{k \le n-1}  P^o_{G^{\prime}}(S_k = y) E^{\Pp}_{G^{\prime}} \left[|C_y| P^y (T_{C_y} = n-k)\right]$ with $E^{\Pp}_{G}\left[|C_x| P^x (T_{C_x} < +\infty)\right]$ second.  

We  first show that 
\begin{equation}\label{sn} 
\left| \widetilde E_{G^{\prime}}^{o,p}[|C_{S_n}|] - E^{\Pp}_{G}[|C_x| ] \right| \le O\left( n^{-(\theta -1)}  \right). 
\end{equation}  

It holds that 
\[ \widetilde E_{G^{\prime}}^{o,p}[|C_{S_n}|] - E^{\Pp}_{G}[|C_x| ] = \sum_{y \in G^{\prime}} P^o_{G^{\prime}}(S_n = y) \left(E^{\Pp}_{G^{\prime}}[|C_y|] - E^{\Pp}_{G}[|C_x| ]\right). \]

If $y \in G^{\prime} \setminus B_{G^{\prime}}(o, N_0 + k)$, then, by the assumption that $G \setminus D$ and $G^{\prime} \setminus D^{\prime}$ are isomorphic,  
\[ E^{\Pp}_{G^{\prime}}[|C_y|, |C_y| < k] = E^{\Pp}_{G}[|C_x|, |C_x| < k], \]
and, 
\[ {\Pp}^{G^{\prime}}(|C_y| \ge k) = {\Pp}^{G}(|C_x| \ge k).\]    

Hence, by using \cite{AV} again, 
\[ \sup_{y \in G^{\prime} \setminus B_{G^{\prime}}(o, N_0 + k)} \left| E^{\Pp}_{G^{\prime}}[|C_y|] - E^{\Pp}_{G}[|C_x| ] \right| \] 
\[ \le C {\Pp}^{G}(|C_x| \ge k) \le C \exp(-ck). \]

Hence, 
\[ \left| \widetilde E_{G^{\prime}}^{o,p}[|C_{S_n}|] - E^{\Pp}_{G}[|C_x| ] \right| \]
\[ \le \left(\sup_{y \in G^{\prime}}  E^{\Pp}_{G^{\prime}}[|C_y|] + E^{\Pp}_{G}[|C_x| ]\right)  P^o_{G^{\prime}}(S_n \in B_{G^{\prime}}(o, N_0 + k))  \]
\[ + \sum_{y \in  G^{\prime} \setminus B_{G^{\prime}}(o, N_0 + k)} P^o_{G^{\prime}}(S_n = y) \left| E^{\Pp}_{G^{\prime}}[|C_y|] - E^{\Pp}_{G}[|C_x| ]\right| \]
\[ \le c\left(\left|B_{G^{\prime}}(o, N_0 + k)\right| n^{-\theta} + \exp(-ck)\right). \]
If we let $k = (\log n)^2$, then, \eqref{sn} follows. 

We then compare $E^{\Pp}_{G}\left[|C_x| P^x (T_{C_x} < +\infty)\right]$ with 
\[ \sum_{k \le n-1}  P^{o}_{G^{\prime}}(S_k = y) E^{\Pp}_{G^{\prime}} \left[|C_y| P^y (T_{C_y} = n-k) \right]. \]
We will show they tend to be arbitrarily close to each other as $n \to \infty$. 

By \cite[Theorem 14.12]{W},   
the Gaussian heat kernel upper bound holds, that is,  
\[ P_{G^{\prime}}^x (S_n = y)  + P_{G^{\prime}}^x (S_{n+1} = y) \le \frac{c}{n^{d/2}}  \exp\left(- c\frac{d_{G^{\prime}}(x, y)^2}{n} \right). \]

By this and Lemma \ref{fm-cri-pb},     
it holds that for {\it any} $y \in G^{\prime}$ and $n > k \ge 1$, 
\[ \left| E^{\Pp}_{G^{\prime}} \left[|C_y| P^y (T_{C_y} = n-k) \right] - E^{\Pp}_{G} \left[|C_x| P^x (T_{C_x} = n-k) \right] \right| \le C (n-k)^{-\theta}, \]
and, 
\[  P^o_{G^{\prime}}(S_k = y) \le \frac{C}{k^{\theta}} \exp\left(- c \frac{d_{G^{\prime}}(o, y)^2}{k}\right). \]

Hence, by using the fact that 
\[ \sum_{k = 1}^{n-1} \left(\frac{1}{k(n-k)}\right)^{\theta} = O(n^{-(\theta-1)}), \]
we have that 
\[ \sum_{y \in B_{G^{\prime}}(o, N_0 + (\log n)^2)} \sum_{k \le n-1}  P^o_{G^{\prime}}(S_k = y) \]
\[ \times \left| E^{\Pp}_{G^{\prime}} [|C_y| P^y (T_{C_y} = n-k) ] - E^{\Pp}_{G} [|C_x| P^x (T_{C_x} = n-k) ] \right| \]
\[ \le C\frac{|B_{G^{\prime}}(o, N_0 + (\log n)^2)|}{n^{\theta - 1}}. \]

If $y \in G^{\prime} \setminus B_{G^{\prime}}(o, N_0 + (\log n)^2)$ and $k \ge n - d_{G^{\prime}}(o,y) + N_0$, then, 
by the exponential decay of the size of the open cluster, 
\[ \left| E^{\Pp}_{G^{\prime}} \left[|C_y| P^y (T_{C_y} = n-k) \right] - E^{\Pp}_{G} \left[|C_x| P^x (T_{C_x} = n-k) \right] \right| \le 2C \exp\left(- c(\log n)^2 \right). \]

Hence, 
\[ \sum_{y \in G^{\prime} \setminus B_{G^{\prime}}(o, N_0 + (\log n)^2)} \sum_{n - d_{G^{\prime}}(o,y) + N_0 \le k \le n-1}  P^o_{G^{\prime}}(S_k = y) \] 
\[ \ \ \times \left| E^{\Pp}_{G^{\prime}} \left[|C_y| P^y (T_{C_y} = n-k) \right] - E^{\Pp}_{G} \left[|C_x| P^x (T_{C_x} = n-k) \right] \right| \]
\[ \le C  \exp(- c(\log n)^2) \sum_{y \in G^{\prime} \setminus B_{G^{\prime}}(o, N_0 + (\log n)^2)} \sum_{n - d_{G^{\prime}}(o,y) + N_0 \le k \le n-1}  P^o_{G^{\prime}}(S_k = y)  \]
\[ \le C n \exp(- c(\log n)^2). \]

Finally we take the sum over $k$ less than $n - d_{G^{\prime}}(o,y) + N_0$. 
\[ \sum_{y \in G^{\prime} \setminus  B_{G^{\prime}}(o, N_0 + (\log n)^2)} \sum_{k \le n - d_{G^{\prime}}(o,y) + N_0}  P^o_{G^{\prime}}(S_k = y) \]
\[ \ \ \ \times \left| E^{\Pp}_{G^{\prime}} [|C_y| P^y (T_{C_y} = n-k) ] - E^{\Pp}_{G} [|C_x| P^x (T_{C_x} = n-k) ] \right| \]
\[ \le \sum_{y \in B_{G^{\prime}}(o, N_0 + n) \setminus  B_{G^{\prime}}(o, N_0 + (\log n)^2)} \sum_{1 \le k \le n - d_{G^{\prime}}(o,y) + N_0} (k(n-k))^{-\theta}  \exp\left(-c\frac{d_{G^{\prime}}(o,y)^2}{k}\right). \] 
\[  \le \sum_{y \in B_{G^{\prime}}(o, N_0 + n) \setminus B_{G^{\prime}}(o, N_0 + (\log n)^2)}  n^{-(\theta-1)}  \exp\left(-c\frac{d_{G^{\prime}}(o,y)^2}{ n - d_{G^{\prime}}(o,y) + N_0}\right). \]
\[ \le C n^{-(\theta-1)} \int_{N_0 + n}^{N_0 + (\log n)^2} t^d \exp\left(-c\frac{t^2}{ n + N_0 -t}\right) dt = C n^{-(\theta-1)}. \]

Thus it holds that 
\[ \left| E^{\Pp}_{G}\left[|C_x| P^x (T_{C_x} < +\infty)\right] - \sum_{y \in G^{\prime}} \sum_{k \le n-1}  P^o_{G^{\prime}}(S_k = y) E^{\Pp}_{G^{\prime}} \left[|C_y| P^y (T_{C_y} = n-k) \right] \right| \]
\[ = O(n^{-(\theta - 1)}).  \]
By this and \eqref{sn}, 
\[ \left| \widetilde E_{G^{\prime}}^{o,p}[U_n - U_{n-1}] - E^{\Pp}_{G}\left[|C_x| P^x (T_{C_x} = +\infty)\right] \right| \le O(n^{-(\theta-1)}). \]

Hence, 
\[  \left| \widetilde E_{G^{\prime}}^{o,p}[U_n] -  c_{G,p} n \right| = O(n^{2-\theta}). \]
Recall $\theta = d/2 > 2$. 
Now we have assertion (ii).  
\end{proof}

\begin{Rem}
(i) We do not yet know about the value of 
$\lim_{n \to \infty} \widetilde E_{G^{\prime}}^{x,p}[U_n] - \widetilde E_{G}^{x,p}[U_n]$. \\ 
(ii) In \cite{Ows}, there is an analog of Theorem \ref{fm} in a  continuous framework.  
The corresponding proof in \cite{Ows} is different from here. 
It does not use the last exit decomposition as in \eqref{last-exit}.   
\end{Rem}


We now consider fluctuation of $\{ \widetilde E^p [U_n] \}_n$.  
Let $\widetilde \Z^d = (\Z^d, E(\widetilde \Z^d))$ be the graph whose vertices and edges are $\Z^d$ and $\{\{x,y\} :  |x-y|_{\infty} = 1\}$.  
$\widetilde \Z^d$ is roughly isometric to $\Z^d$. 

\begin{Lem}
For any non-empty finite subset $A \subset \Z^3$,   
\begin{equation}\label{cap-ineq} 
\ca_{\Z^3}(A) < \ca_{\widetilde \Z^3}(A).  
\end{equation} 
\end{Lem} 

\begin{proof} 
Let 
\[ E({\Z^3}^{\prime}) := E(\widetilde \Z^3) \setminus E(\Z^3) = \left\{ \{x,y\} : |x-y|_{\infty} = 1 < |x-y| \right\}. \]
Then, 
${\Z^3}^{\prime} := (\Z^3, E({\Z^3}^{\prime}))$ is an infinite connected vertex-transitive graph and satisfies $V(d)$ for some $d > 2$. 
Hence, by \cite[Corollary 4.16]{W}, $(\Z^3, E({\Z^3}^{\prime}))$ is transient. 
Hence, 
\[ \inf\left\{\sum_{\{x,y\} \in E({\Z^3}^{\prime})} (f(x) - f(y))^2 : f = 1 \textup{ on $A$, $\textup{supp}(f)$ is compact} \right\} > 0.  \] 
By using this and \eqref{ca-re}, 
\[ \ca_{\widetilde \Z^3}(A) = \inf\left\{\sum_{\{x,y\} \in E(\widetilde \Z^3)} (f(x) - f(y))^2 : f = 1 \textup{ on $A$, $\textup{supp}(f)$ is compact} \right\} \]
\[ \ge \inf\left\{\sum_{\{x,y\} \in E(\Z^3)} (f(x) - f(y))^2 : f = 1 \textup{ on $A$, $\textup{supp}(f)$ is compact} \right\} \]
\[ + \inf\left\{\sum_{\{x,y\} \in E({\Z^3}^{\prime})} (f(x) - f(y))^2 : f = 1 \textup{ on $A$, $\textup{supp}(f)$ is compact} \right\} \]
\[ > \inf\left\{\sum_{\{x,y\} \in E(\Z^3)} (f(x) - f(y))^2 : f = 1 \textup{ on $A$, $\textup{supp}(f)$ is compact} \right\} = \ca_{\Z^3}(A).  \]
\end{proof}

\begin{Lem}\label{fluc-lem}
For any $p \in [0, p_T (\widetilde\Z^3))$, 
\begin{equation*}
c_{\Z^3, \ p} < c_{\widetilde\Z^3, \ p}.
\end{equation*}
\end{Lem}

\begin{proof} 
$\widetilde \Z^d$ also has a structure of a Cayley graph of $\Z^d$ with a generating set different from the nearest-neighbor $\Z^d$.   
By \eqref{cp-cap}, 
it suffices to show that 
\[ E^{\Pp}_{\Z^3} \left[\ca(C_0)\right]  < E^{\Pp}_{\widetilde\Z^3} \left[\ca(C_0)\right].  \]
We regard Bernoulli bond percolation on $\Z^3$ as Bernoulli bond percolation on $\widetilde\Z^3$ such that all of edges $\{\{x,y\} : |x-y|_{\infty} = 1 < |x-y| \}$ declared to be closed.    
Then, by \eqref{cap-ineq},   
\[ E^{\Pp}_{\Z^3} \left[\ca_{\Z^3}\left(C_0 \right)\right] < E^{\Pp}_{\Z^3} \left[\ca_{\widetilde\Z_3}\left(C_0 \right)\right]  \le E^{\Pp}_{\widetilde\Z^3} \left[\ca_{\widetilde\Z^3}\left(C_{0}\right)\right]. \]
\end{proof}

\begin{proof}[Proof of Theorem \ref{fluc}] 
As an outline level, we follow the proof of \cite[Theorem 1.3]{O14}, but here we need to deal with unboundedness of $\cup_{i \le n} C_{S_i}$. 

Let $p < p_T (\widetilde\Z^3)$.   
Let $G_1 := \Z^3$.
For a strictly increasing sequence of natural numbers $(M_k)_k$, 
let $G_{k+1} := G(M_1, \cdots, M_k)$ be the graph such that $M_i \le |x|_{\infty} \le M_{i+1}$ has the structure of $\widetilde \Z^3$ if $i < k$ is odd, and, 
has the structure of $\Z^3$ if $i < k$ is even, and, $M_k \le |x|_{\infty}$ has the structure of $\widetilde \Z^3$ if $k$ is odd, and, 
has the structure of $\Z^3$ if $k$ is even. 
Here $M_0 := 0$.   
Let $G_{\infty}$ be the graph such that $M_i \le |x|_{\infty} \le M_{i+1}$ has the structure of $\widetilde \Z^3$ if $i$ is odd, and, 
has the structure of $\Z^3$ if $i$ is even. 
The set of vertices of $G_{\infty}$ is $\Z^3$, and, it is a subgraph of $\widetilde \Z^3$.  
All $G_{k}, k \le +\infty$, are roughly isometric to $\Z^3$.   

We now specify $(M_k)_k$. 
We define a strictly increasing sequence $(n_k)_k$. 
Let $n_0 := 1$. 
Let $n_k > \exp(n_{k-1})$ such that 
\begin{equation}\label{basic-conv-1} 
\left| \frac{\widetilde E^{0,p}_{G_k} \left[U_{n_k}\right]}{n_k} - c \right| \le \exp(-n_{k-1}).  
\end{equation} 
In the above, $c = c_{\Z^3, p}$ if $k$ is odd, and, 
$c = c_{\widetilde\Z^3, p}$ if $k$ is even. 
We assume that $M_{k} > n_k + \exp(n_k)$ for each $k$. 

If  $k$ is sufficiently large, then, by the exponential decay of the size of the cluster, 
\[ {\Pp}^{G_k}\left( \left(\bigcap_{x \in B_{\infty} (0, n_k)} \left\{ |C_x| \le \exp(n_k) \right\}\right)^c \right) \le \exp(- n_k).\]    
Let $i > k$. 
Since  the event $\cap_{x \in B_{G_k, \infty} (0, n_k)} \left\{ |C_x| \le \exp(n_k) \right\}$ is determined only by the state of edges in $B_{G_i, \infty} (0, n_k) (= B_{G_k, \infty} (0, n_k))$, 
we have that 
\[ E^{\Pp}_{G_k}\left[U_{n_k},  \bigcap_{x \in B_{G_k, \infty} (0, n_k)} \left\{ |C_x| \le \exp(n_k) \right\}\right] 
=  E^{\Pp}_{G_i}\left[U_{n_k}, \bigcap_{x \in B_{G_i, \infty} (0, n_k)} \left\{ |C_x| \le \exp(n_k) \right\}\right]. \]

Hence, if $i \ge k$, then, 
\begin{align*} 
\left|  \widetilde E^{0,p}_{G_{i}}[U_{n_k}] - \widetilde E^{0,p}_{G_{i+1}}[U_{n_k}] \right| 
&= \Big| \widetilde E^{0,p}_{G_{i}}\left[U_{n_k}, \ \cup_{x \in B_{\infty} (0, n_k)} \left\{ |C_x| > \exp(n_k) \right\}\right] \\
& \ - \widetilde E^{0,p}_{G_{i+1}}\left[U_{n_k}, \ \cup_{x \in B_{\infty} (0, n_k)} \left\{ |C_x| > \exp(n_k) \right\}\right] \Big|  \\
&\le n_k \exp(-n_{i}).  
\end{align*} 

It holds that for each fixed $k$,
\[ \lim_{i \to \infty} \widetilde E^{0,p}_{G_{i}}[U_{n_k}] = \widetilde E^{0,p}_{G_{\infty}}[U_{n_k}]. \]
Hence, 
\[ \left| \widetilde E^{0,p}_{G_{k}}[U_{n_k}] - \widetilde E^{0,p}_{G_{\infty}}[U_{n_k}] \right| \le n_k \sum_{i \ge k} \exp(-n_{i}).  \]
By this and \eqref{basic-conv-1}, 
\begin{equation*}
\left| \frac{\widetilde E^{0,p}_{G_{\infty}}[U_{n_k}]}{n_k} - c \right| \le \sum_{i \ge k-1} \exp(-n_{i}).  
\end{equation*} 
In the above, $c = c_{\Z^3, p}$ if $k$ is odd, and, 
$c = c_{\widetilde\Z^3, p}$ if $k$ is even. 
Thus 
\begin{equation}\label{fluc-fluc-pre}
\liminf_{n \to \infty} \frac{\widetilde E^{o,p} [U_{n}]}{n} < \limsup_{n \to \infty} \frac{\widetilde E^{o,p} [U_{n}]}{n}.  
\end{equation}  
holds for $G = G_{\infty}$, $x = 0$ and $p < p_T (\widetilde \Z^3)$. 

We then replace $p_T (\widetilde \Z^3)$ above with $p_T (G)$.    
We will show that for a $(M_k)_k$ suitably chosen,  $p_{T}(G_{\infty}) = p_T (\widetilde \Z^3)$.     
Since $G_{\infty}$ is a subgraph of $\widetilde \Z^3$, $p_{T}(G_{\infty}) \ge p_T (\widetilde \Z^3)$.  
Now it suffices to show 
\begin{equation}\label{pT-ineq}
p_{T}(G_{\infty}) \le p_T (\widetilde \Z^3). 
\end{equation} 

Let $p > p_T (\widetilde \Z^3)$.  
Since, it holds that $p_T (\widetilde \Z^3) = p_H (\widetilde \Z^3)$ by \cite{AV},    
it holds that for each $k$, 
there exists a vertex $x_k$ such that 
$|x|_{\infty} = M_{2k-1} + 1$, and furthermore,  
with positive probability under ${\Pp}^{G_{2k}}$, 
there exists an infinite path which does not hit any vertex of $B_{\infty}(0, M_{2k-1})$. 

Denote by $G_{2k} \setminus B_{\infty}(0, M_{2k-1})$ the graph obtained by delating all edges of $B_{\infty}(0, M_{2k-1})$ from $G_{2k}$.   
It is an infinite connected simple graph.  
It holds that 
\[ E^{\Pp}_{G_{2k} \setminus B_{\infty}(0, M_{2k-1})} \left[  |C_{x_k}| \right] = +\infty. \] 
Hence, if $M_{2k}$ is sufficiently large, 
then, it holds that   
\[  E^{\Pp}_{G_{2k} \cap B_{\infty}(0, M_{2k})} \left[  |C_{x_k}| \right] \ge E^{\Pp}_{G_{2k} \cap B_{\infty}(0, M_{2k}) \setminus B_{\infty}(0, M_{2k-1})} \left[  |C_{x_k}| \right] \]
\[\ge \sum_{l = M_{2k-1}}^{M_{2k}-M_{2k-1}}  {\Pp}^{G_{2k} \setminus B_{\infty}(0, M_{2k-1})} \left( |C_{x_k}|  > l \right) \ge  p^{- 2d (1 + M_{2k-1})}. \]
By repeating this argument, 
and by noting 
\[ G_{2k} \cap B_{\infty}(0, M_{2k}) = G_{i} \cap B_{\infty}(0, M_{2k}), \ i \ge 2k,\]   
\[  E^{\Pp}_{G_{\infty}} \left[  |C_{x_k}| \right]  = +\infty.  \]
Hence $p_{T}(G_{\infty}) \le p$ and hence \eqref{pT-ineq} holds. 
Thus \eqref{fluc-fluc-pre} holds for $G = G_{\infty}$, $x = 0$ and $p < p_T (G)$.

Finally, we show the almost sure equalities of \eqref{fluc-fluc}.   
By \eqref{fm-as}, 
for each $k$, there exists $l_k$ such that 
\[ \widetilde P^{0,p}_{G_k} \left(\left| \frac{U_{l_k}}{l_k} - \lim_{n \to \infty} \frac{\widetilde E^{0,p}_{G_k} \left[U_{n}\right]}{n} \right| \ge \frac{1}{2^k} \right) \le \frac{1}{4^k}.  \]
If we take a sufficiently large $M_k > l_k$ for each $k$, then,   
\[ {\Pp}^{G_k} \left( B_{\infty}(0, l_k) \leftrightarrow B_{\infty} (0, M_k)^c \right) \le \frac{1}{4^k}. \]
It holds that 
\[ \widetilde P^{0,p}_{G_{\infty}} \left(\left| \frac{U_{l_{2k}}}{l_{2k}} -  c_{\widetilde\Z^3,  p}\right|  > \frac{1}{2^k} \right) \le \frac{1}{4^k},  \]
and, 
\[ \widetilde P^{0,p}_{G_{\infty}} \left(\left| \frac{U_{l_{2k+1}}}{l_{2k+1}} - c_{\Z^3,  p} \right| > \frac{1}{2^{k+1}} \right) \le \frac{1}{4^{k+1}}.  \]
Hence, 
\[ \lim_{k \to \infty} \frac{U_{l_{2k}}}{l_{2k}} =  c_{\widetilde\Z^3,  p} >  c_{\Z^3,  p} = \lim_{k \to \infty} \frac{U_{l_{2k+1}}}{l_{2k_1}}, \ \textup{$\widetilde P^{0,p}_{G_{\infty}}$-a.s.}  \]
Thus, the proof of \eqref{fluc-fluc} is completed.  
\end{proof}

\begin{Rem}
As in the proof of \cite[Theorem 1.3]{O14},  we can replace $\Z^3$ and $\widetilde \Z^3$ with the  regular trees of degrees $3$ and $4$, respectively. 
\end{Rem} 


\section{Two-dimensional lattice}

This section is devoted to consider the case that $G = \Z^2$.  

\begin{proof}[Proof of Theorem \ref{lowdim-dis} (a) - (c)] 
(a) \cite[Lemma 3.1]{JP70} and \cite[(2.15)]{DE} imply that 
\begin{equation*}
E^{P^0}[R_n] = \frac{n}{\log n}\pi + O\left(\frac{n}{(\log n)^2}\right). 
\end{equation*} 
By \eqref{Okada24} and \eqref{boundary}, 
\begin{equation*}
\widetilde E^{0, p}\left[U_n - R_n\right] = O\left(\frac{n}{(\log n)^2}\right). 
\end{equation*} 
Now \eqref{asy-2d} follows from these estimates.  

(b)  
\begin{equation}\label{Var-ineq} 
\va_{x,p}(U_n) \le 2\left(\va_x (R_n)  + \va_{x,p}(U_n - R_n)\right). 
\end{equation}  
By \eqref{bd-always}, the Cauchy-Schwarz inequality, and \eqref{L-2nd}, 
\begin{equation}\label{var-difference}
\va_{p}(U_n - R_n) \le E^{P^0}[L_n^2] E^{\Pp}\left[|C_0|^2\right] = O\left(\frac{n^2}{(\log n)^4}\right). 
\end{equation} 
By \cite[Theorem 4.2]{JP72}, 
\begin{equation}\label{JP72-var}
\va(R_n) = O\left(\frac{n^2}{(\log n)^4}\right). 
\end{equation}
\eqref{Var-2d} follows from \eqref{Var-ineq} \eqref{var-difference}, and \eqref{JP72-var}.     

(c) By applying an interpolation argument for $U_n$ as in the proof of \cite[Theorem 3.1]{JP72}, 
the almost sure convergence of \eqref{lim-2-dis-as} follows from \eqref{Var-2d} and \eqref{asy-2d}.  

Now we show the $L^q$-convergence for $1 \le q < +\infty$. 
We can assume that $q$ is an integer without loss of generality.     
If we show that for {\it any} $q$, 
\begin{equation}\label{Lq-Un}
\lim_{n \to \infty} \left( \frac{\log n}{n} \right)^q \widetilde E^p [U_n^q] = \pi^q, 
\end{equation}
then, 
$\{((\log n) U_n / n)^q\}_n$ are uniformly integrable for {\it any} $q$.  
Now for each $q$, 
the $L^q$-convergence follows from this, the $\widetilde P^p$-a.s. convergence of $(\log n) U_n / n$ to $\pi$, and the fact that $U_n \ge 0$. 

The rest of this proof are devoted to show \eqref{Lq-Un}. 
First we show the following:  
\begin{Lem}\label{2d-range-moment}
We have that\footnote{The corresponding result for the volume of the Wiener sausage follows from \cite[Corollarie 2-2]{Le86AOP}.}
\begin{equation}\label{Lq-Rn}
\lim_{n \to \infty} \left( \frac{\log n}{n} \right)^q E^{P^0} \left[R_n^q \right] = \pi^q. 
\end{equation}
\end{Lem}

\begin{proof}[Proof of Lemma \ref{2d-range-moment}]
By \cite[Theorem 4]{DE}, 
\begin{equation}\label{DE4} 
\lim_{n \to \infty} \frac{\log n}{n} R_n = \pi, \textup{ $P^0$-a.s.} 
\end{equation} 
By this and Fatou's lemma,
\[ \liminf_{n \to \infty} \left( \frac{\log n}{n} \right)^q E^{P^0} \left[R_n^q \right] \ge \pi^q. \] 
Hence, it suffices to show that
\[ \limsup_{n \to \infty} \left( \frac{\log n}{n} \right)^q E^{P^0} \left[R_n^q \right] \le \pi^q. \] 

Since 
\[ E^{P^0} \left[R_n^q \right] = \sum_{x_1, \dots, x_q \in [-n,n]^2} P^0 \left( \bigcap_{1 \le i \le q } \{H_{x_i} \le n\} \right),  \] 
it suffices to show that  
\begin{equation}\label{limsup-Rnq-2}
\limsup_{n \to \infty} \left( \frac{\log n}{n} \right)^q  \sum_{r = 1}^{q} \sum_{x_1, \cdots, x_r \in [-n,n]^2; \ \textup{distinct}} P^0 \left(\bigcap_{1 \le i \le r } \{H_{x_i} \le n\} \right) \le \pi^q.
\end{equation}

Let $1 \le r \le q$. 
Let $\Pi_r$ be the permutation group on $\{1,2, \dots, r\}$. 
By the Markov property and the translation invariance, 
if $x_1, \cdots, x_r$ are distinctive, 
\[ P^0 \left(\bigcap_{1 \le i \le r } \{H_{x_i} \le n\} \right) = \sum_{k_1, \dots k_r \in [0,n]; \ \textup{distinct}} P^0 \left(\bigcap_{1 \le i \le  r} \{H_{x_i} = k_i \} \right) \]
\[ = \sum_{\sigma \in \Pi_r} \sum_{1 \le k_{\sigma(1)} < \cdots < k_{\sigma(r)} \le n} P^0 \left(\bigcap_{1 \le i \le r} \{H_{x_i} = k_i \} \right). \]

Let $\sigma(0) = 0$ and $k_0 = 0$. 
By the Markov property, 
\[ \sum_{\sigma \in \Pi_r} \sum_{1 \le k_{\sigma(1)} < \cdots < k_{\sigma(r)} \le n} P^0 \left(\bigcap_{1 \le i \le r} \{H_{x_i} = k_i \} \right)\]
\[ \le \sum_{\sigma \in \Pi_r} \sum_{1 \le k_{\sigma(1)} < \cdots < k_{\sigma(r)} \le n} \prod_{i=1}^{q} P^0 \left(H_{x_{\sigma(i)} - x_{\sigma(i-1)}} = k_{\sigma(i)} - k_{\sigma(i-1)} \right) \]

Hence, 
\[  \sum_{x_1, \cdots, x_r \in [-n, n]^2; \ \textup{distinct}} P^0 \left(\bigcap_{1 \le i \le r} \{ H_{x_i} \le n \}\right) \]
\[ \le \sum_{\sigma \in \Pi_r} \sum_{1 \le k_{\sigma(1)} < \cdots < k_{\sigma(r)} \le n}  \sum_{x_1, \cdots, x_r \in [-n, n]^2} \prod_{i=1}^{q} P^0 \left(H_{x_{\sigma(i)} - x_{\sigma(i-1)}} = k_{\sigma(i)} - k_{\sigma(i-1)} \right) \]
\[ \le  \sum_{\sigma \in \Pi_r} \sum_{1 \le k_{\sigma(1)} < \cdots < k_{\sigma(r)} \le n} \prod_{i=1}^{q} \sum_{x \in \Z^2 \setminus \{0\}} P^0 \left(H_x = k_{\sigma(i)} - k_{\sigma(i-1)} \right)\]
\[ = \sum_{k_1, \dots k_r \in [1,n]; k_1 + \cdots + k_r \le n} \prod_{i=1}^{r} \sum_{x \in \Z^2  \setminus \{0\}} P^0 \left(H_x = k_{i}\right)  \] 

Let 
\[ f(k) = \sum_{x \in \Z^2  \setminus \{0\}} P^0 \left(H_x = k\right) \]
and
\[ g(k) = \sum_{x \in \Z^2} P^0 \left(H_x = k\right). \]
If $k \ge 1$, then, $f(k) = g(k)$. 
We have that $f(0) = 0$ and $g(0) = 1$. 
By taking sum over $r$,
\[ \sum_{r = 1}^{q} \sum_{k_1, \dots k_r \in [1,n]; k_1 + \cdots + k_r \le n} \prod_{i=1}^{r} f(k_i) \le \sum_{k_1, \dots k_q \in [0,n]; k_1 + \cdots + k_q \le n} \prod_{i=1}^{q} g(k_i) \]
\[\le  \left(\sum_{x \in \Z^2} P^0 \left(H_x \le n\right) \right)^q = E^{P^0}[R_n]^q. \] 
By \cite[Theorem 1]{DE}, 
\[ \lim_{n \to \infty} \frac{\log n}{n} E^{P^0}[R_n] = \pi.\]  
Now \eqref{limsup-Rnq-2} holds. 
\end{proof}

Now we return to the proof of part (c) of Theorem \ref{lowdim-dis}. 
By Lemma \ref{2d-range-moment}, 
in order to show \eqref{Lq-Un}, 
it suffices to show that 
\begin{equation}\label{Lq-UnRn}
\lim_{n \to \infty} \left( \frac{\log n}{n} \right)^q \widetilde E^{0,p} \left[(U_n -R_n)^{q} \right] = 0.    
\end{equation}

By \eqref{bd-always} and the H\"older inequality, 
\begin{align}\label{Holder} 
\widetilde E^{0,p} \left[(U_n -R_n)^{q} \right] \\
&\le E^{P^0} \left[ E^{\Pp} \left[\left(\sum_{x \in \partial R_n} |C_x| \right)^{q} \right]\right] \\
&\le E^{P^0}\left[L_n^q\right] E^{\Pp}\left[|C_0|^q\right] \notag\\
&\le E^{\Pp}\left[|C_0|^q\right] E^{P^0}\left[ \left(\frac{L_n}{R_n}\right)^{2q}\right]^{1/2} E^{P^0}[R_n^{2q}]^{1/2}. 
\end{align} 

By \eqref{Okad-as} and the dominated convergence theorem, 
\[ \lim_{n \to \infty}  E^{P^0}\left[ \left(\frac{L_n}{R_n}\right)^{2q}\right] = 0. \]

Now \eqref{Lq-UnRn} follows from this, \eqref{Lq-Rn} and \eqref{Holder}.   
\end{proof}

\begin{Rem}
We can give an alternative proof of the a.s. convergence part of part (c) of Theorem \ref{lowdim-dis} as follows. 
Since $\partial R_n$ surrounds the origin and  $L_n = O\left(n/(\log n)^2\right)$, 
by using the isoperimetric inequality for subsets of $\Z^2$,   
there is a path from the origin to  $\partial R_n$ whose length is of order $O\left(n^{1/2}/ \log n \right)$. 
Hence we can apply  Fontes-Newman \cite[Theorem 4]{FN}\footnote{\cite[Theorem 4]{FN} is stated for site percolation, 
but, as in the proof of  Grimmett-Piza \cite[Lemma 6]{GP}, it holds for the case of bond percolation.},    
\begin{equation*}
\frac{1}{L_n} \sum_{x \in \partial R_n}  |C_x| < +\infty,  \ \textup{ $\Pp$-a.s.}
\end{equation*}  
Therefore, 
by using this, \eqref{bd-always}  and \eqref{DE4}, 
\[ \lim_{n \to \infty} \frac{U_n}{R_n} = 1,  \ \textup{ $\Pp$-a.s.}\]
Hence we have the a.s. convergence part of \eqref{lim-2-dis-as}. 
\end{Rem}

\begin{Prop}[Volume of intersections]\label{inter}  
Let $S^1$ and $S^2$ be two independent simple random walks on $\Z^2$ starting at $0$. 
Let \[ I_n := \left| \left(\bigcup_{i \in [0,n]} C_{S^1_i}\right)  \cap \left(\bigcup_{i \in [0,n]} C_{S^2_i}\right) \right|.  \]
Let 
\[ f_p (n) := E^{P^{0,0} \otimes \Pp}\left[ I_n \right], \ 0 \le p < p_T (\Z^2).  \]
Then, 
\begin{equation}\label{inter-conv} 
\lim_{n \to \infty} \frac{(\log n)^2}{n} (f_p (n) - f_0 (n)) = 0. 
\end{equation} 
\end{Prop}

\begin{Rem}\label{indep} 
The fact that $U_{n_1, n_2}$ and $U_{n_3, n_4}$ are not independent for $n_1 < n_2 \le n_3 < n_4$ 
will be an obstacle also for the proof of the estimate corresponding to \cite[(6.u)]{Le86CMP}. 
However, the situation is different in the {\it deterministic} case.  
In the case, we can show the statement corresponding to Theorem \ref{lowdim-dis} (c) by using the statement corresponding to Proposition \ref{inter}, 
as in the proof of \cite[Lemme 6.2]{Le86CMP}. 
See the Appendix for details. 
\end{Rem}

\begin{proof}
In this proof, if we consider $P^{z}$ for $z \in \R^2$, then we take the integer part of $z$ (i.e. $y \in \Z^2$ such that $z \in y + [0,1)^2$).   
By the translation invariance of percolation, 
\[ E^{P^{x,y} \otimes \Pp}\left[ I_n - I_{0,n} \right] = \sum_{z \in \Z^2} E^{\Pp} \left[ P^0 (T_{C_z} \le n)^2  - P^0 (T_{z} \le n)^2 \right] \]
\[ = \sum_{z \in \Z^2} E^{\Pp} \left[ P^{z} (T_{C_0} \le n < T_0) P^{z} (T_{C_0} \le n)  + P^{z} (T_{0} \le n) P^{z} (T_{C_0} \le n < T_0) \right]. \]
\[  \le 2  E^{\Pp} \left[ \sum_{z \in \Z^2}\sum_{x,y \in C_0}  P^{z} (T_{x} \le n < T_0) P^{z} (T_{y} \le n) \right]\]

Using the Cauchy-Schwarz inequality repeatedly, 
\[ \frac{(\log n)^2}{n} E^{\Pp}\left[\sum_{z \in \Z^2} \sum_{x,y \in C_0}  P^{z} (T_{x} \le n < T_0) P^{z} (T_{y} \le n)\right] \]
\begin{equation}\label{inter-step-1} 
\le \left(E^{\Pp}\left[\sum_{x \in C_0} \sum_{z \in \Z^2} \frac{(\log n)^2}{n}  P^{z} (T_{x} \le n < T_0)^2\right] E^{\Pp}[| C_0| ] \sum_{z \in \Z^2} \frac{(\log n)^2}{n} P^{z} (T_{0} \le n)^2 \right)^{1/2}. 
\end{equation} 

By change of variables, 
\[ \sum_{z \in \Z^2} \frac{(\log n)^2}{n}  P^{z} (T_{x} \le n < T_0)^2 = \int_{\R^2} (\log n)^2  P^{x + n^{1/2}z} (T_{x} \le n < T_0)^2  dz.  \]
Here $dz$ denotes the Lebesgue measure on $\R^2$.   
As in \cite[Theoreme 3.5]{Le86CMP}, 
let 
\[ f_2 (r) := \max\{0, - \log r \} + r^{-2} 1_{\{r \ge 1/2\}}. \]
Let $\delta > 0$. 
Then, by the Gaussian heat kernel estimates for the simple random walk on $\Z^2$,   
\[ \limsup_{n \to \infty} (\log n)  P^{x + n^{1/2}z} (n < T_{0} \le n(1 + \delta)) \le \log (1+\delta). \] 
\[ \limsup_{n \to \infty} (\log n)  P^{n^{1/2}z} (T_{0} \le n) P^{x} (T_{0} \ge n\delta) \le f_2(|z|) \limsup_{n \to \infty} P^{x} (T_{0} \ge n\delta) = 0. \]
Since $\delta > 0$ is taken arbitrarily, 
it holds that for each $z \ne 0$,  
\[ \lim_{n \to \infty} (\log n)  P^{x + n^{1/2}z} (T_{x} \le n < T_0) = 0. \] 

We remark that for each $z \in \R^2$, 
\[ (\log n)  P^{x + n^{1/2}z} (T_{x} \le n < T_0)  \le (\log n)  P^{x + n^{1/2}z} (T_{x} \le n) \]
\[= (\log n)  P^{n^{1/2}z} (T_{0} \le n)  \le f_2 (|z|)^2. \]

Now by the dominated convergence theorem, for each $x \in \Z^2$, 
\[ \lim_{n \to \infty} \sum_{z \in \Z^2} \frac{(\log n)^2}{n}  P^{z} (T_{x} \le n < T_0)^2 = 0. \] 

By this and $p < p_T (\Z^2)$, we have 
\[ \sum_{x \in C_0} \sum_{z \in \Z^2} \frac{(\log n)^2}{n}  P^{z} (T_{x} \le n < T_0)^2 \to 0, \ \textup{ $\Pp$-a.s.}\] 

It holds that 
\[ \sum_{x \in C_0} \sum_{z \in \Z^2} \frac{(\log n)^2}{n}  P^{z} (T_{x} \le n < T_0)^2 \le |C_0| \int_{\R^2} f_2 (|z|)^2 dz. \] 

By the dominated convergence theorem, 
\begin{equation}\label{inter-small}
\lim_{n \to \infty} E^{\Pp}\left[\sum_{x \in C_0} \sum_{z \in \Z^2} \frac{(\log n)^2}{n}  P^{z} (T_{x} \le n < T_0)^2\right] = 0. 
\end{equation} 

On the other hand, 
\[  \frac{(\log n)^2}{n}  P^{z} (T_{x} \le n)^2 \le \frac{1}{n} \sum_{y \in \Z^2} f_2 \left(|y| / \sqrt{n} \right)^2 \le C \int_{\R^2}  f_2 (|z|)^2 dz. \]
Here $C$ is a positive constant. 
Hence, 
\[ E^{\Pp}\left[\sum_{x \in C_0} \sum_{z \in \Z^2} \frac{(\log n)^2}{n}  P^{z} (T_{x} \le n)^2\right] \le C E^{\Pp}[|C_0|] \int_{\R^2}  f_2 (|z|)^2 dz < +\infty. \]

By this, \eqref{inter-step-1}, and \eqref{inter-small}, 
we have \eqref{inter-conv}. 
\end{proof}


\section{One-dimensional graphs}

In this section, we deal with one-dimensional graphs. 

\begin{proof}[Proof of Theorem \ref{unbdd}]
Figure 1 after this proof would facilitate understanding the following construction.  

Construct $G = G(\{a_n, b_n\}_n)$ as follows. 
First, prepare the line graph $\left(\mathbb{N}, \{\{n, n+1\} : n \in \mathbb{N}\}\right)$, 
and then attach $b_n$ new vertices each of which is connected by an edge to each of the vertices $a_n$ for each $n$.  
(Here we assume that $\mathbb{N}$ contains $0$.) 
Suitable choices of values of $a_n$ and $b_n$ will lead the desired result. 
It holds that $T_{a_n} \ge a_n \to \infty, n \to \infty$, $P^0$-a.s. 
By \cite[Theorem 2.12]{W}, $G$ is recurrent.  
Hence, $T_{a_n} < +\infty$, $P^0$-a.s. 

We define $a_n$ and $b_n$ by induction on $n$. 
Assume that $a_i, b_i, 1 \le i \le n-1$, are given. 
Let $c_0 := 0$ and $c_{k} := \sum_{i \le k} (a_i + b_i)$, $1 \le k \le n-1$.     
Then we define $a_n$ and $b_n$ satisfying that  
\begin{equation}\label{a_n_1}
a_n > \exp\left( \left(\sum_{i \le n-1} (a_i + b_i) \right)^2 \right), 
\end{equation}
and 
\begin{equation}\label{b_n_1}
b_n = \textup{the integer part of} \ n^{-4} p^{-2a_n}. 
\end{equation}  

(a) It suffices to show that 
\[ E^{\Pp}[|C_0|] < +\infty  \textup{ for any $p < 1$.} \]    
We have that 
\[ \Pp\left(|C_0| > c_{n-1} + l\right) \le p^{a_{n-1} + l}, \ \ 0 \le l < a_{n}. \]
\[   \Pp\left(|C_0| > c_{n-1} + l\right) \le p^{a_{n}}, \ \ a_{n} \le l \le a_n + b_{n}. \]
\eqref{b_n_1} imply that $b_n p^{a_n} = O(n^{-2})$.  
By this and \eqref{a_n_1}, 
\[ E^{\Pp}[|C_0|] = \sum_{n} \sum_{l = 0}^{a_n + b_n} \Pp(|C_0| > c_{n-1} + l) < +\infty. \]

(b) Let $V_n$ be the number of open edges adjacent to $a_n$.  
Then, \[ \sum_{i=1}^{n} V_i \ge U_{T_{a_n}} - R_{T_{a_n}} \ge V_n.\]   
By a large deviation estimate for the sum of the Bernoulli trials,  
\begin{align*}
\widetilde P^{0,p} \left(U_{T_{a_n}} > p b_n/2\right) &\ge \Pp(V_n > p b_n/2) \\
&\ge 1 - \exp(- c b_n p/2). 
\end{align*} 
By the Borel-Cantelli lemma, 
\begin{equation}\label{bcu} 
\liminf_{n \to \infty} \frac{U_{T_{a_n}}}{b_n} \ge \frac{p}{2} > 0, \  \textup{ $\widetilde P^{0, p}$-a.s.} 
\end{equation}

By \cite[Lemma 4.1.1(v)]{Ku} and \eqref{a_n_1},  
\begin{align*}\label{}
\sum_{n} P_{G}^{0}\left( T_{a_n}  > a_n^2 \log a_n \right) 
&\le \sum_{n} a_n^{-2} (\log a_n)^{-1} E_{G}^{P^0}\left[T_{a_n}\right] \\
&\le \sum_{n} 2 a_n^{-1}  (\log a_n)^{-1} R^{G}_{\textup{eff}}\left(0, G \setminus [0, a_n] \right) \\
&= \sum_{n} 2  (\log a_n)^{-1}  < +\infty.  
\end{align*}
 
By using this, \eqref{b_n_1} and the Borel-Cantelli lemma,   
\[ P^{0}\left( T_{a_n}< C_p (\log b_n)^2 \log\log b_n \textup{ infinitely many $n$}\right) = 1. \]
Now \eqref{loglogUn} follows from this and \eqref{bcu}. 

(c) Recall \eqref{a_n_1}.  
By considering the random walk which goes {\it only} in the right direction at every time and calculating the probability, 
we have that for some positive constant $c$ the following holds for any $n$: 
\begin{equation}\label{p_a_n} 
P^0 (T_{a_n} = a_n) \ge \left(2^{a_n - (n-1)}  \prod_{i \le n-1} (b_i + 2)\right)^{-1} \ge c a_n^{-1} 2^{-a_n}. 
\end{equation} 

Using 
\[ E^{\Pp}[(V_n - E^{\Pp}[V_n])^2] = p(1-p) b_n, \]
\[ \widetilde E^{0,p} \left[(U_{T_{a_n}} - R_{T_{a_n}} - V_n)^2\right]  \le \left( 1 +  \sum_{i \le n-1} b_i  \right)^2,\]
 \eqref{b_n_1} and the Minkowski inequality,   
we have that the following holds $P^0$-a.s.: 
\begin{align*} 
E^{\Pp}\left[\left(U_{T_{a_n}} - E^{\Pp}[U_{T_{a_n}} ]\right)^2\right]^{1/2} &= E^{\Pp}\left[ \left(U_{T_{a_n}} - R_{T_{a_n}} - E^{\Pp}[U_{T_{a_n}} -R_{T_{a_n}}] \right)^2\right]^{1/2}\\
&\ge  E^{\Pp}[(V_n - E^{\Pp}[V_n])^2]^{1/2} -  \widetilde E^{0,p} \left[(U_{T_{a_n}} - R_{T_{a_n}} - V_n)^2\right]^{1/2} \\
&\ge  (p(1-p) b_n)^{1/2} - 1 - \sum_{i \le n-1} b_i  \\
&\ge \frac{1}{2}  (p(1-p) b_n)^{1/2}. 
\end{align*}

By this and \eqref{p_a_n},   
\begin{align*} 
\va_{0,p} (U_{a_n}) 
&\ge \frac{1}{4} E^{\Pp}[(V_n - E^{\Pp}[V_n])^2] P^0 (T_{a_n}= a_n) \\
&= \frac{1}{2} p(1-p) b_n^2 P^0(T_{a_n}= a_n) \\
&\ge c n^{-3} p^{-2a_n} 2^{-a_n} \ge c (2p^2)^{-a_n}.     
\end{align*} 

\eqref{VarUn} follows from this.  
\end{proof}

\begin{center}

\begin{tikzpicture}
 \draw (0,0) -- (1,0) -- (2,0) -- (3,0) -- (4,0) -- (5,0) -- (6,0) -- (7,0) -- (8,0) -- (9,0) -- (10,0);
\fill (1,0) circle (2pt);
\fill (2,0) circle (2pt);
\fill (3,0) circle (2pt);
\fill (4,0) circle (2pt);
\fill (5,0) circle (2pt);
\fill (6,0) circle (2pt);
\fill (7,0) circle (2pt);
\fill (8,0) circle (2pt);
\fill (9,0) circle (2pt);
\node[below=4mm] () at (2,0) {$a_{n}$};
\draw (8,-1) -- (8,0) -- (8,1); 
\draw (7.14, -0.5) -- (8,0) -- (8.86, 0.5);
\draw (7.14, 0.5) -- (8,0) -- (8.86, -0.5);
\draw (7.5, -0.86) -- (8,0) -- (8.5, 0.86);
\draw (7.5, 0.86) -- (8,0) -- (8.5, -0.86);
 \draw (19/7, 5/7) -- (2,0) -- (9/7,-5/7); 
 \draw (19/7, -5/7) -- (2,0) -- (9/7,5/7); 
 \node[below=11mm] () at (8,0) {$a_{n+1}$};
 \fill (8,-1) circle (2pt);
 \fill (8,1) circle (2pt);
 \fill (7.14, -0.5) circle (2pt);
  \fill (8.86, 0.5) circle (2pt);
   \fill (7.14, 0.5) circle (2pt);
  \fill (8.86, -0.5) circle (2pt);
   \fill (7.5, -0.86) circle (2pt);
  \fill (8.5, 0.86) circle (2pt);
     \fill (7.5, 0.86) circle (2pt);
  \fill (8.5, -0.86) circle (2pt);
   \fill (19/7, 5/7) circle (2pt);
  \fill (9/7,-5/7) circle (2pt);
   \fill (19/7, -5/7) circle (2pt);
  \fill (9/7,5/7) circle (2pt);

\end{tikzpicture}

Rough figure of the graph $G$\\

\end{center}

\begin{Rem}\label{Rem-unbdd}
For the graph $G$ in the above proof, we have the following:\\
(i) 
\[ \sup_{x \in G} E^{\Pp}[|C_x|] = +\infty, \ p > 0.  \] 
On the other hand,   
\[ E^{\Pp}[|C_x|] < +\infty, \ x \in G, \ p \in [0, p_T (G)). \] 
(ii) 
If $p \in (0, 1/2)$, then \eqref{lim-rec-0} {\it fails} for $G$. 
Indeed, by \eqref{p_a_n}, 
\[ \widetilde E^{0,p}[U_{a_n}] \ge \widetilde E^{0,p}\left[U_{T_{a_n}}, T_{a_n} = a_n\right]  \ge E^{\Pp}[V_n] P^0 (T_{a_n} = a_n)  \ge c p  \frac{b_n}{a_n 2^{a_n}}.  \]
Hence, 
\[ (2p)^{E^{P^0}[R_{a_n}]}  \widetilde E^{0,p}[U_{a_n}] \ge c \frac{p^{a_n + 1} }{a_n} b_n. \]
By this, $0 < p < 1/2$, \eqref{a_n_1} and \eqref{b_n_1},  
\[ \lim_{n \to \infty} (2p)^{E^{P^0}[R_{a_n}]}  \widetilde E^{0,p}[U_{a_n}] = +\infty. \]
\end{Rem}

\subsection{Remarks}

\begin{Thm}[One-dimensional case]\label{lowdim-dis-1}
Assume that $p < p_T (G)$.  
If $G = \Z$, then, 
the laws of $U_n/\sqrt{n}$ under $\widetilde P^{p}$ converges to the law of $\max_{0 \le t \le 1} B_t - \min_{0 \le t \le 1} B_t$ as $n \to \infty$, 
where $(B_t)_t$ denotes Brownian motion. 
\end{Thm}

If $p = 0$, then, this follows from \cite[Theorem 6.1]{JP72}.    

\begin{proof}
By using \eqref{bd-always} and $\partial R_n = \left\{\min_{i \le n} S_i,  \max_{i \le n} S_i \right\}$,    
we have 
\begin{equation}\label{1-bd} 
0 \le U_n - R_n \le \left| C_{\max_{i \le n} S_i} \right| + \left| C_{\min_{i \le n} S_i} \right|. 
\end{equation}

By the translation invariance of percolation, 
it holds that 
\[ \widetilde P^p \left( \left|C_{\max_{i \le n} S_i}\right| > n^{1/4} \right) = \widetilde P^p \left( \left|C_{\min_{i \le n} S_i}\right| > n^{1/4} \right) = \Pp \left( |C_{0}| > n^{1/4} \right) \le 2p^{n^{1/4}/2}. \]

Hence, by the Borel-Cantelli lemma, 
\[ \frac{U_n - R_n}{n^{1/2}} \to 0, \ n \to \infty, \  \textup{ $\widetilde P^p$-a.s.}  \]

By \cite[Theorem 6.1]{JP72}, 
\[ \frac{R_n}{\sqrt{n}} \Rightarrow \max_{0 \le t \le 1} B_t - \min_{0 \le t \le 1} B_t, \ n \to \infty, \ \textup{ in the law of $P^0$.} \]
\end{proof}

\begin{Rem}
Assume that 
\begin{equation}\label{sup-sub}
\sup_{x \in G} E^{\Pp} \left[|C_x|\right] < +\infty,  
\end{equation} 
and for any vertex $x$,  
\begin{equation}\label{strong-rec} 
\lim_{n \to \infty} \frac{\log R_n}{\log n} = \lim_{n \to \infty} \frac{\log E^{P^x}[R_n]}{\log n} = \alpha \in (0,1), \textup{ \ $P^x$-a.s.} 
\end{equation}
Then, we can show that for any vertex $x$,  
\[ \lim_{n \to \infty} \frac{\log U_n}{\log n} = \lim_{n \to \infty} \frac{\log \widetilde E^{x,p}[U_n]}{\log n} = \alpha, \textup{ \ $\widetilde P^{x,p}$-a.s.} \]
By using \eqref{boundary} and \eqref{sup-sub}, 
\[ \lim_{n \to \infty} \frac{\log \widetilde E^{x,p}[U_n]}{\log n} = \alpha. \]
Now we show that 
\[ \lim_{n \to \infty} \frac{\log U_n}{\log n} = \alpha, \textup{ \ $\widetilde P^{x,p}$-a.s.} \]
In order to show this, it suffices to show that 
\begin{equation}\label{upper-sr} 
\limsup_{n \to \infty} \frac{\log U_n}{\log n} \le \alpha, \textup{ \ $\widetilde P^{x,p}$-a.s.} 
\end{equation} 

Let $\epsilon > 0$. 
Then, by \eqref{strong-rec}, 
three exists a constant $C$ depending on $x$ such that for any $n \ge 1$, 
\[ \widetilde P^{x,p}\left(U_n \ge n^{\alpha + \epsilon}\right) \le \sup_{z \in G} E^{\Pp} \left[|C_z|\right] \frac{E^{P^x}[R_n]}{n^{\alpha + \epsilon}} 
\le \sup_{z \in G} E^{\Pp} \left[|C_z|\right]  \frac{C}{n^{\epsilon /2}}. \]
Hence for sufficiently large integer $k$, 
\[ \sum_{n \ge 1}  \widetilde P^{x,p}\left(U_{n^k} \ge n^{k(\alpha + \epsilon)}\right) < +\infty. \]
By the Borel-Cantelli lemma, 
\[ \limsup_{n \to \infty} \frac{U_{n^k}}{ n^{k(\alpha + \epsilon)}} \le 1, \textup{ \ $\widetilde P^{x,p}$-a.s.} \]
Since $U_n$ is non-decreasing with respect to $n$, 
\[ \limsup_{n \to \infty} \frac{U_{n}}{ n^{\alpha + \epsilon}} \le 1, \textup{ \ $\widetilde P^{x,p}$-a.s.} \]
Since $\epsilon$ is taken arbitrarily, 
we have \eqref{upper-sr}
\end{Rem}

It is known that \eqref{strong-rec} holds for a large class of recurrent fractal graphs. 
See \cite{BJKS, KN09, KM, HHH}. 
It would be interesting to determine whether \eqref{sup-sub} holds on any $p \in [0, p_T (G))$.      
This fails for some ``singular" graph $G$. 
See Remark \ref{Rem-unbdd}.  
Let $G$ be a $d \ge 2$-dimensional Sierpinski gasket graph, or, $2$-dimensional Sierpinski carpet graph.   
Bernoulli percolation on such graphs has been considered by Shinoda \cite{Sh96, Sh02, Sh03}, Kumagai \cite{K97}, Higuchi-Wu \cite{HW08}.   
We conjecture that \eqref{sup-sub} holds for such fractal graphs. 


\section{Positive and negative exponentials}

\begin{proof}[Proof of Theorem \ref{posi}]
Since $\exp(x) \le 1 + x + x^2 \exp(x) /2$ holds for $x \ge 0$,  
as in \cite[(3.1)]{H}, it holds that for each $n \ge 1$, 
\begin{equation}\label{posi-upper}
\Lambda_p (\theta) \le \frac{\theta}{n} \widetilde E^p [U_{n}] +  \frac{\theta^2}{2n} \widetilde E^p \left[U_{n}^2 \exp(\theta U_{n})\right]   
\end{equation}

(i) By \eqref{posi-upper},  for each fixed $n \ge 1$,  
\[ \limsup_{\theta \to 0} \frac{\Lambda_p (\theta) }{\theta} \le \frac{\widetilde E^p [U_n]}{n}. \] 

Hence, 
\[ \limsup_{\theta \to 0} \frac{\Lambda_p (\theta) }{\theta} \le c_p. \] 

By Jensen's inequality, $\Lambda_p (\theta) \ge \theta c_p$.  
Thus it holds that 
\[ \lim_{\theta \to 0} \frac{\Lambda_p (\theta) }{\theta} = c_p. \] 

(ii) 
Assume $G = \Z^2$. 
By the Cauchy-Schwartz inequality,  
\[ U_n^2 \exp(\theta U_n) \le R_n \left(\sum_{x \in S[0,n]} |C_{x}|^2\right) \prod_{y \in S[0,n]} \exp(\theta |C_y|). \] 

By the Cauchy-Schwartz inequality and the translation invariance,   
\begin{align*} 
E^{\Pp}\left[|C_{x}|^2 \prod_{y \in S[0,n]} \exp(\theta |C_y|) \right] &\le E[|C_0|^4]^{1/2} E^{\Pp}\left[\prod_{y \in S[0,n]} \exp(2\theta |C_y|) \right]^{1/2}  \\
&\le E^{\Pp}[|C_0|^4]^{1/2} E^{\Pp}\left[\exp(2\theta R_n |C_0|) \right]^{1/2}  
\end{align*} 

Hence, 
\begin{equation}\label{2nd-upper}
\widetilde E^p \left[U_n^2 \exp(\theta U_n)\right]  \le E^{\Pp}[|C_0|^4]^{1/2} E^{P^0}\left[ R_n^2 E^{\Pp}\left[\exp(2\theta R_n |C_0|) \right]^{1/2} \right].  
\end{equation}

Let 
\[ n(\theta) := \theta E \left[R_{1/\theta}\right]^2, \textup{ and } \ f_p (\theta) := \frac{n (\theta)}{\widetilde E^p \left[U_{n(\theta)}\right]}. \] 

Then, by recalling \eqref{asy-2d}, 
we have that as $\theta \to 0$
\[ n(\theta) \sim \frac{\pi^2}{ \theta (\log (1/\theta))^2}, \]
and 
\[ f_0 (\theta) \sim \frac{1}{\pi} \log n(\theta) \sim \log (1/\theta) - 2 \log\log (1/\theta). \] 

By using \eqref{posi-upper}, \eqref{2nd-upper}, \eqref{boundary} and the fact that $R_n \le n$, 
it holds that 
\begin{align*}
\frac{f_p (\theta)}{\theta} \Lambda_p (\theta)  
&\le 1 + \frac{\theta E^{\Pp}\left[|C_0|^4\right]^{1/2}}{2\widetilde E^p [U_{n(\theta)}]}  E^{P^0}\left[ R_{n(\theta)}^2 E^{\Pp}\left[\exp(2\theta R_{n(\theta)} |C_0|) \right]^{1/2} \right] \\  
&\le 1 + \frac{\theta E^{\Pp}\left[|C_0|^4\right]^{1/2}}{2\widetilde E^p [U_{n(\theta)}]}  E^{P^0}\left[ R_{n(\theta)}\right] n(\theta) E^{\Pp}\left[\exp(2\theta n(\theta) |C_0|) \right]^{1/2} \\  
&\le 1 + \frac{1}{2}E^{\Pp}\left[|C_0|^4\right]^{1/2}  \theta n(\theta) E^{\Pp}\left[\exp(2\theta n(\theta) |C_0|) \right]^{1/2}.  
\end{align*} 

By this and $\lim_{\theta \to 0} \theta n(\theta) = 0$,
it holds that 
\begin{equation}\label{fp-sup}
\limsup_{\theta \to 0} \frac{f_p (\theta)}{\theta} \Lambda_p (\theta) \le 1.  
\end{equation}

By \eqref{lim-rec-0} and $\lim_{\theta \to 0} n(\theta) = +\infty$,
it holds that 
\begin{equation}\label{fp-ratio}
\lim_{\theta \to 0} \frac{f_p (\theta)}{f_0 (\theta)} = \lim_{\theta \to 0} \frac{\widetilde E^p [U_{n(\theta)}]}{E^{P^0}[R_{n(\theta)}]} = 1. 
\end{equation}

By \cite{H}, 
\begin{equation}\label{Hamana}
\lim_{\theta \to 0} \frac{f_0 (\theta)}{\theta} \Lambda_0 (\theta) = 1. 
\end{equation}

By using \eqref{fp-sup}, \eqref{fp-ratio}, \eqref{Hamana} and the fact that $\Lambda_p (\theta) \ge \Lambda_0 (\theta)$, 
we have the desired result. 

If $G = \Z$, then,  
by \eqref{1-bd}, 
it is easy to see that 
\[ \Lambda_p (\theta) = \Lambda_0 (\theta), \ 0 < \theta < \theta_c(p). \] 
Thus we have assertion (ii).  
\end{proof}

\begin{proof}[Proof of Theorem \ref{Dis-DV}]
As in \cite[(1.2)]{DV79}, let 
\[ k(\theta, d) := \theta^{2/(d+2)} \frac{d+2}{2} \left(\frac{2\gamma_d}{d} \right)^{d/(d+2)} > 0, \]
where $\gamma_d$ is the lowest eigenvalue of the Laplacian $-(1/2) \Delta$ for $B(0,1)$ with zero boundary values.
We let $b_n := n^{1/(d+2)}$. 

Let $\mathcal{O}$ be the set of bounded open subsets of $\mathbb{R}^d$. 
Let \[ \textup{dist}(A_1, A_2) := \inf\{\|x-y\| : x \in A_1, y \in A_2\}, \ \  A_1, A_2 \subset \mathbb{R}^d. \]
For each $H \in \mathcal{O}$, 
let $\lambda(H)$ be the smallest eigenvalue of the Laplacian with zero boundary conditions for $H$. 

For any $H_i \in \mathcal{O}, i = 0,1,2,$ satisfying that $H_0 \subset \overline{H_0} \subset H_1 \subset \overline{H_1} \subset H_2$,      
\begin{align*}
\widetilde  E^{p}[\exp(-\theta U_{n})] &\ge \exp(-\theta b_n^d |H_2|) \widetilde P^{p}\left( \bigcup_{x \in S[0,n]} C_x \subset b_n H_2 \right) \\
&\ge \exp(-\theta b_n^d |H_2|)  P^{0}(S[0,n] \subset b_n H_1)  \left(1- \sum_{x \in b_n H_1} \Pp\left( C_x \not\subset b_n H_2\right)\right) 
\end{align*} 

Using the exponential decay of sizes of clusters of subcritical percolations on $\Z^d$ and $\overline{H_1} \subset H_2$,
\[ \frac{1}{b_n^d} \log \left(1- b_n^d |H_1| \Pp\left( \textup{diam}(C_x) > b_n \textup{dist}(H_1, H_2^c) \right)\right) \to 0, \ \ n \to \infty. \]

Therefore,  
\[ \liminf_{n \to \infty} \frac{\log \widetilde  E^{p}[\exp(-\theta U_{n})]}{b_n^d} \ge -\theta |H_2| +  \liminf_{n \to \infty} \frac{\log  P^{0}(S[0,n] \subset b_n H_1)}{b_n^d}.  \]   

By \cite[Lemma 5.1]{DV79}, 
\[ \liminf_{n \to \infty} \frac{\log  P^{0}(S[0,n] \subset b_n H_1)}{b_n^d} \ge -\lambda(H_0).  \] 

Since we can make $|H_2 \setminus H_0|$ arbitrarily small, 
\[ \liminf_{n \to \infty} \frac{\log \widetilde  E^{p}[\exp(-\theta U_{n})]}{b_n^d} \ge -\theta |H_2| -\lambda(H_0) \] 
\[\ge -\inf_{H \in \mathcal{O}} \theta |H| +\lambda(H). \]

By \cite[(5.5) and arguments after Theorem 1]{DV79}, 
\[ k(\theta, d) = \theta^{2/(d+2)} \frac{d+2}{2} \left(\frac{2}{d} \inf_{O \in \mathcal{O}, |O| = 1} \lambda(O)\right)^{d/(d+2)} = -\inf_{H \in \mathcal{O}} \theta |H| +\lambda(H). \] 
Hence, 
\[ \liminf_{n \to \infty} \frac{\log \widetilde  E^{p}[\exp(-\theta U_{n})]}{b_n^d} 
\ge k(\theta, d) = \lim_{n \to \infty} \frac{\log E^{P^0}[\exp(-\theta R_{n})]}{b_n^d}. \]

\eqref{Dis-DV-int} follows from this and the fact that $R_n \le U_n$. 
\end{proof}


\appendix 

\section{A discrete analog of the Wiener sausage} 

In this section, we consider extensions of results concerning the range of random walk, which was stated by \cite[Section 4]{S64}.      
We assume that $G$ is $\Z^d$, $d \ge 3$ and $A$ is a finite subset of $\Z^d$ containing the origin.   
Here we let $\{S_i\}_i$ be the simple random walk.  
We omit most proofs of results in this section and refer the readers to suitable references.

\begin{Def}
For $A \subset \Z^d$ and $n \ge m \ge 0$, let 
\[ W_{m,n} (A) := \bigcup_{i \in [m,n]} (S_i + A), \textup{ and, } U_{m,n} (A) := \left| W_{m,n} (A) \right|. \]
Let $W_n (A) := W_{0,n}(A)$ and $U_n (A) := U_{0,n}(A)$.            
\end{Def}
 
If $A = \{0\}$, then, $U_n (A) = R_n$. 
The process $\{U_n (A)\}_n$  is also a natural analog of the Wiener sausage.   
\cite[Section 4]{S64} states that the strong law holds for $d \ge 3$.   
Port \cite{P65, P66}, Port-Stone \cite{PS68, PS69} considered asymptotics for means of the volumes in connection with interacting particle systems.  

\begin{Thm}[{\cite[Section 4]{S64}, \cite[Problem 14, Chapter 6]{S76}}]
If $G = \Z^d, d \ge 3$, then, 
\[ \lim_{n \to \infty} \frac{U_n (A)}{n} = \ca(A), \ \ \textup{  $P$-a.s.}  \]
\end{Thm}

As was noted in \cite[Problem 14, Chapter 6]{S76},  this gives an alternative definition of the capacity of a set.  
By using this, in the same manner as in the proof of \eqref{LLN},     
we can also show that for any $q < +\infty$, 
\[ \lim_{n \to \infty} \frac{U_n (A)}{n} = \ca(A), \ \ \textup{ in $L^q$. }  \]
The last exit decomposition as in \eqref{last-exit} may not work well in this case in a direct manner.

\begin{Thm}\label{VaCLT}
If $d \ge 4$, then,\\
(i) 
\[ c_{\textup{var}}(A) := \lim_{n \to \infty} \frac{\va(U_n (A))}{n} \]
exists and is positive. \\
(ii) 
\[ \frac{U_n (A) - E[U_n (A)]}{\sqrt{c_{\textup{var}} (A) n}} \Rightarrow N(0,1), \ n \to \infty, \textup{ in law.} \]
Here and henceforth $N(0,1)$ is the normal distribution with mean zero and covariance one. 
\end{Thm}

We are not sure whether the value of $c_{\textup{var}}(A)$ differs depending on the choice of $A$.   
We can show this by imitating the proof of \cite[Lemma 4.2]{Le88} with some alterations.     

\begin{Thm}\label{2-CLT-det}  
Let $d = 2$. 
Let $\gamma$ be the renormalized self-intersection local time of two-dimensional Brownian motion given in \cite[Theoreme 6.1]{Le86CMP}. 
Formally, 
\[ \gamma := \int_{0}^{1} \int_{0}^{t} \delta_{0} (B_t - B_s) dsdt - E\left[ \int_{0}^{1} \int_{0}^{t} \delta_{0} (B_t - B_s) dsdt \right],  \]
where $(B_t)_t$ is the standard 2-dimensional Brownian motion and $\delta_0$ be the delta function. 
More precisely, 
\[ \gamma := \lim_{\epsilon \to 0} \left[ \int_{0}^{1} \int_{0}^{t} \varphi_{\epsilon} (B_t - B_s) dsdt - E\left[ \int_{0}^{1} \int_{0}^{t}  \varphi_{\epsilon}(B_t - B_s) dsdt \right] \right],  \]
where $\varphi_{\epsilon}(x) = \epsilon^{-2} \varphi(x/\epsilon)$ is a suitable approximation of the identity, where $\varphi$ is a smooth nonnegative function in the Schwartz class whose integration is one. 
Then, the following weak convergence holds:   
\[ \frac{(\log n)^2}{n} \left( U_{n}(A) - E[U_n (A)] \right) \Rightarrow -2\pi^2 \gamma, \  n \to \infty, \textup{ in law.} \]
\end{Thm}

We can show this assertion in the same manner as in the proof of \cite[Theoreme 6.1]{Le86CMP}.   

\begin{Rem}\label{hk-alt}
We can give an alternative proof of \eqref{hk}. 
By the translation invariance,  
\[ \widetilde E^p [U_n] = E^{\Pp}\left[ E^{P^0}\left[ U_n (C_0) \right] \right].  \] 
By the proof of \cite[Theorem 3.1]{P66}, 
\[ E[U_n (A)] - n \ca(A) = \sum_{k = 1}^{n} \sum_{x \in A} P^x (k < T_{A} < +\infty). \]
Hence, by noting $p < p_T (\Z^d)$ and \eqref{cp-cap}, 
\[ \widetilde E^p [U_n] - n c_p = \sum_{k = 1}^{n} \sum_{A \subset \Z^d, \textup{ finite}} \Pp(C_0 = A) \sum_{x \in A} P^x (k < T_{A} < +\infty). \]
\eqref{hk} follows from this.
\end{Rem} 

\subsection{Intermediate process}

Let $\overline{U}^{(p)}_{m, n} := E^{\Pp}[U_{m,n}], \ m \le n$, and let $\overline{U}^{(p)}_{n} := \overline{U}^{(p)}_{0, n}$.   
$\{ \overline{U}^{(p)}_{n} \}_n$ is more similar to $\{U_n (A)\}_n$ than $\{U_n\}_n$.  
Contrary to the case of $\{U_n (A)\}_{n}$, 
we can consider this on general graphs which do not possess a group structure.  
An application of Theorem \ref{DisLLN} is the following: 

\begin{Cor}[SLLN]
Let $G$ be a vertex-transitive graph, $o$ be a vertex of $G$, and $c_p$ be as in \eqref{cp}. 
Then, \\
 (i) \[ \lim_{n \to \infty} \frac{\overline{U}^{(p)}_{n}}{n} = c_p, \ P^o \textup{-a.s.} \] 
 (ii) \[ \lim_{n \to \infty} \frac{E^{P^{o}}[U_{n}]}{n} = c_p, \ \Pp\textup{-a.s.} \]
\end{Cor}

\begin{proof} 
We show (i). 
Applying Liggett's theorem \cite{Li} to $\{E^{\Pp}[U_{m,n}]\}_{m,n}$,  
there is a random variable $Y$ such that 
\[ \frac{E^{\Pp}[U_{n}]}{n} \to Y, \ n \to \infty, \ \ P^{o}\textup{-a.s.} \]  
By Theorem \ref{DisLLN} (i), 
\[ \lim_{n \to \infty} \frac{U_{n}}{n} = c_p,  \textup{ in } L^{1}(\widetilde P^{o, p}).\] 
By Fubini's theorem, 
\[ E^{\Pp}\left[\left|\frac{U_{n}}{n} - c_p\right|\right] \to 0, \ n \to \infty, \  \textup{  in $L^{1}(P^{o})$.}  \] 
By taking a subsequence $(n_{k})_{k}$,  
\[ \frac{E^{\Pp}[U_{n_{k}}]}{n_{k}} \to c_p, \ k \to \infty, \  \textup{ $P^{o}$-a.s.} \] 
Hence we obtain $c_p = Y$, $P^{o}$-a.s.   
We can show (ii) in the same manner.  
\end{proof}

It holds that for $n_1 < n_2 \le n_3 < n_4$, $\overline{U}^{(p)}_{n_1, n_2}$ and $\overline{U}^{(p)}_{n_3, n_4}$ are independent, 
therefore Le Gall's approach \cite{Le86CMP, Le88} is adaptable. 
Recall Remark \ref{indep}.  

Now we have the following CLTs for $\{\overline{U}^{(p)}_{n}\}_n$.  

\begin{Thm}[Variances and CLT]
Let $d \ge 4$. Then,\\
(i) 
\[ c_{\textup{var}, p} := \lim_{n \to \infty} \frac{\va\left(\overline{U}^{(p)}_{n}\right)}{n} \]
exists and is positive.  \\
(ii) 
\[ \frac{\overline{U}^{(p)}_{n} - E \left[\overline{U}^{(p)}_{n} \right]}{\sqrt{c_{\textup{var}, p} n}} \Rightarrow N(0,1), \ n \to \infty, \textup{ in law.} \]
\end{Thm}

We can show this by imitating the proof of \cite[Lemma 4.2]{Le88} with some alterations.

\begin{Thm}
Let $d = 2$. 
Let $\gamma$ be the self-intersection local time of two-dimensional Brownian motion. 
Then, 
\[ \frac{(\log n)^2}{n} \left( \overline{U}^{(p)}_{n} - E \left[ \overline{U}^{(p)}_{n} \right] \right) \Rightarrow -2\pi^2 \gamma, \ n \to \infty, \textup{ in law.} \]
\end{Thm}

We can show this assertion in the same manner as in the proof of \cite[Theoreme 6.1]{Le86CMP}.

\begin{Rem} 
(i) $\va_{p}(U_n) \ge \va\left(\overline{U}^{(p)}_{n}\right)$. \\ 
(ii) We are not sure properties for $c_{\textup{var}, p}$ as a function of $p$.  
\end{Rem}

\vspace{1\baselineskip}  

{\it Acknowledgement.}  
The author thanks to N. Kubota for proposing questions of this kind and discussions.  
He also thanks to R. Fukushima, P. Mathieu, I. Okada and M. Takei for suggestions and comments.  
This work was supported by  Grant-in-Aid for Research Activity Start-up (15H06311) and for JSPS Research Fellows (24.8491, 16J04213).


\begin{thebibliography}{99}
\bibitem[AV08]{AV} T. Antunovi\'c and I. Veseli\'c, Sharpness of the phase transition and exponential decay of the subcritical cluster size for percolation on quasi-transitive graphs, 
{\it J. Stat. Phys.} 130, (2008) 983-1009. 
\bibitem[AS15]{AS1} A. Asselah and B. Schapira, Boundary of the Range of Transient Random Walk, {\it Probab. Theory and Related Fields}, 168 (2017) 691-719.  
\bibitem[AS16]{AS2} A. Asselah and B. Schapira, Moderate deviations for the range of a transient random walk: path concentration, {\it Ann. Sci. \'Ec. Norm. Sup\'er. (4)} 50 (2017) 755-786. 
\bibitem[BJKS08]{BJKS} M. T. Barlow, A. A. J\'arai, T. Kumagai, and G. Slade, Random walk on the incipient infinite cluster for oriented percolation in high dimensions. {\it Comm. Math. Phys.} 278 (2008) 385-431. 
\bibitem[BR06]{BR} B. Bollob\'as and O. Riordan, {\it Percolation}, Cambridge University Press, 2006. 
\bibitem[BL83]{BL83} H. Br\'ezis, and E. Lieb, A relation between pointwise convergence of functions and convergence of functionals, {\it Proc. Amer. Math. Soc.} 88 (1983), no. 3, 486-490.  
\bibitem[D10]{D} R. Diestel, Graph theory, Fourth edition, Springer, 2010. 
\bibitem[DV79]{DV79} M. D. Donsker and S. R. S. Varadhan, On the number of distinct sites visited by a random walk, {\it Comm. Pure. Appl. Math.} 32 (1979) 721-747. 
\bibitem[DE51]{DE} A. Dvoretzky and P. Erd\"os, Some problems on random walk in space, {\it Proc. Second Berkeley Symp. on Math. Stat. and Prob.}, 353-368, Univ. of California Press, 1951. 
\bibitem[FH17]{FH17} R. Fitzner and R. van der Hofstad,  Mean-field behavior for nearest-neighbor percolation in $d > 10$, {\it Electron. J. Probab.}  22 (2017), no. 43, 1-65.
\bibitem[FN93]{FN} L. Fontes and C. M. Newman, First passage percolation for random colorings of $\Z^d$, {\it Ann. Appl. Prob.}, 3 (1983) 746-762; Erratum, 4, (1994) 254.  
\bibitem[Gi08]{Gi} L. R. Gibson, The mass of sites visited by a random walk on an infinite graph, {\it Elec. Comm. Probab.} 13 (2008) 1257-1282.
\bibitem[Gr99]{Gr} G. R. Grimmett, \textit{Percolation}, Springer-Verlag, 2nd ed., 1999.
\bibitem[GP97]{GP} G. R. Grimmett and M. S. T. Piza, Decay of correlations in random-cluster model, {\it Comm. Math. Phys}. 189 (1997) 465-480.  
\bibitem[H01]{H} Y. Hamana, Asymptotics of the Moment Generating Function for the Range of Random Walks, {\it J. Theoret. Probab.} 14 (2001) 189-197. 
\bibitem[HHH14]{HHH} M. Heydenreich, R. van der Hofstad, T. Hulshof, Random Walk on the High-Dimensional IIC, {\it Comm. Math. Phys.} 329 (2014) 57-115. 
\bibitem[HW08]{HW08} Y. Higuchi and X.-Y. Wu, Uniqueness of the critical probability for percolation in the two-dimensional Sierpi\'nski carpet lattice,  {\it Kobe J. Math.} 25 (2008), no. 1-2, 1-24. 
\bibitem[JP70]{JP70} N. C. Jain and W. E. Pruitt, The range of recurrent random walk in the plane, {\it Z. Wahr. Verw. Gebiete} 16 (1970) 279-292.  
\bibitem[JP72]{JP72} N. C. Jain and W. E. Pruitt, The range of random walk. {\it Proceedings of the Sixth Berkeley Symposium on Mathematical Statistics and Probability (Univ. California, Berkeley, Calif., 1970/1971)}, Vol. III, 31-50. Univ. California Press, Berkeley, Calif., 1972.
\bibitem[KS63]{KS} H. Kesten and F. Spitzer, Ratio theorems for random walks I, {\it J. Analyse Math.} 11 (1963) 285-321.  
\bibitem[KN09]{KN09} G. Kozma,  A. Nachmias, The Alexander-Orbach conjecture holds in high dimensions. {\it Invent. Math.} 178 (2009), no. 3, 635-654.
\bibitem[K97]{K97} T. Kumagai, Percolation on pre-Sierpinski carpets. {\it New trends in stochastic analysis (Charingworth, 1994), 288-304, World Sci. Publ., River Edge, NJ}, 1997.
\bibitem[K14]{Ku} T. Kumagai, Random walks on disordered media and their scaling limits, \'Ecole d'\'Et\'e de Probabilit\'es de Saint-Flour XL-2010,  Lecture Notes in Mathematics, 2101, Springer, 2014. 
\bibitem[KM08]{KM} T. Kumagai and J. Misumi, Heat kernel estimates for strongly recurrent random walk on random media. {\it J. Theoret. Probab.} 21 (2008), no. 4, 910-935. 
\bibitem[La96]{La} G. Lawler, {\it Intersections of random walks}, Birkh\"auser, 1996.  
\bibitem[Le86CMP]{Le86CMP} J.-F. Le Gall, Propriet\'es d'intersection des marches al\'eatoires. {\it Comm. Math. Phys.} 104 (1986) 471-507. 
\bibitem[Le86AOP]{Le86AOP} J.-F. Le Gall, Sur la saucisse de Wiener et les points multiples du mouvement brownien. {\it Ann Probab.} 14 (1986) 1219-1244. 
\bibitem[Le88]{Le88} J.-F. Le Gall, Fluctuation results for the Wiener sausage, {\it Ann. Probab.} 16 (1988)  991-1018.
\bibitem[Li85]{Li} T. M. Liggett, An improved subadditive ergodic theorem, {\it Ann. Probab.} 13 (1985) 1279-1285.   
\bibitem[Okad14]{Okad} I. Okada, The inner boundary of random walk range, {\it J. Math. Soc. Japan} 68 (2016) 939-959.  
\bibitem[Okam14]{O14} K. Okamura, On the range of random walk on graphs satisfying a uniform condition, 
{\it ALEA Lat. Am. J. Probab. Math. Stat.} 11 (2014), no. 1, 341-357.
\bibitem[Okam17]{O} K. Okamura, Enlargement of subgraphs of infinite graphs by Bernoulli percolation,  {\it Indag. Math. (N.S.)} 28 (2017), no. 4, 832-853.
\bibitem[Okam-ws]{Ows} K. Okamura, Long time behavior of the volume of the Wiener sausage on Dirichlet spaces, in preparation.   
\bibitem[P65]{P65} S. C. Port, Limit theorems involving capacities for recurrent Markov chains, {\it J. Math. Anal. Appl.} 12 (1965) 555-569.
\bibitem[P66]{P66} S. C. Port,  Limit theorems involving capacities, {\it J. Math. Mech.} 15 (1966) 805-832. 
\bibitem[PS68]{PS68} S. C. Port and C. J. Stone,  Hitting times for transient random walks. {\it J. Math. Mech.} 17 (1968) 1117-1130. 
\bibitem[PS69]{PS69} S. C. Port and C. J. Stone,  Potential theory of random walks on Abelian groups, {\it Acta Math.} 122 (1969) 19-114.
\bibitem[Sh96]{Sh96} M. Shinoda, Percolation on the pre-Sierpinski gasket, {\it Osaka J. Math.} 33 (1996), no. 2, 533-554.
\bibitem[Sh02]{Sh02} M. Shinoda, Existence of phase transition of percolation on Sierpinski carpet lattices, {\it J. Appl. Probab.} 39 (2002), no. 1, 1-10. 
\bibitem[Sh03]{Sh03} M. Shinoda, Non-existence of phase transition of oriented percolation on Sierpinski carpet lattices. {\it Probab. Theory Related Fields} 125 (2003), no. 3, 447-456.
\bibitem[Sp64]{S64} F. Spitzer, Electrostatic capacity, heat flow and Brownian motion, {\it Z. Wahr. Verw. Gebiete}, 3 (1964) 110-121.
\bibitem[Sp76]{S76} F. Spitzer, {\it Principles of Random Walk}, second edition, Springer, 1976.
\bibitem[W00]{W} W. Woess, {\it Random walks on infinite graphs and groups}, Cambridge University Press, Cambridge, 2000.  
\end{thebibliography}
\end{document}